\documentclass{amsart}
\author{Dilip Raghavan}
\thanks{Part of this article was written when the first author was a visitor at the Institute of Mathematical Sciences, Chennai. He thanks them for their kind hospitality.}
\author{Juris Stepr{\=a}ns}
\thanks{Both authors partially supported by NSERC}
\date{\today}
\title{On weakly tight families}
\usepackage{amssymb, amsmath, amsthm, mathrsfs, enumerate, amsfonts, latexsym, bbm}
\def\polhk#1{\setbox0=\hbox{#1}{\ooalign{\hidewidth
    \lower1.5ex\hbox{`}\hidewidth\crcr\unhbox0}}}
\newtheorem{Theorem}{Theorem}

\newtheorem{Lemma}[Theorem]{Lemma}

\newtheorem{conj}[Theorem]{Conjecture}
\newtheorem{Question}[Theorem]{Question}
\theoremstyle{definition}
\newtheorem{Def}[Theorem]{Definition}

\theoremstyle{remark}

\renewcommand{\qedsymbol}{$\dashv$}
\newcommand{\forces}{\Vdash}
\newcommand{\restrict}{\upharpoonright}
\newcommand{\forallbutfin}{{\forall}^{\infty}}
\newcommand{\existsinf}{{\exists}^{\infty}}
\renewcommand{\c}{\mathfrak{c}}
\renewcommand{\b}{\mathfrak{b}}
\renewcommand{\d}{{\mathfrak{d}}}
\newcommand{\s}{\mathfrak{s}}
\renewcommand{\ae}{{\mathfrak{a}}_{\mathfrak{e}}}

\newcommand{\x}{{\mathfrak{x}}}
\newcommand{\y}{{\mathfrak{y}}}
\newcommand{\p}{{\mathfrak{p}}}
\renewcommand{\a}{{\mathfrak{a}}}
\renewcommand{\[}{\left[}
\renewcommand{\]}{\right]}
\renewcommand{\P}{\mathbb{P}}
\newcommand{\PP}{\mathcal{P}}

\newcommand{\lc}{\left|}
\newcommand{\rc}{\right|}

\newcommand\ZFC{\mathrm{ZFC}}

\newcommand\CH{\mathrm{CH}}
   
\newcommand{\BS}{{\omega}^{\omega}}

\DeclareMathOperator{\otp}{otp}
\DeclareMathOperator{\cov}{cov}
\DeclareMathOperator{\dom}{dom}

\DeclareMathOperator{\cf}{cf}
\newcommand{\Pset}{\mathcal{P}}
\newcommand{\M}{\mathcal{M}}

\newcommand{\B}{\mathscr{B}}
\newcommand{\A}{{\mathscr{A}}}
\newcommand{\C}{{\mathscr{C}}}

\newcommand{\GG}{{\mathcal{G}}}

\newcommand{\cube}{{\[\omega\]}^{\omega}}

\newcommand{\I}{{\mathcal{I}}}
\newcommand{\F}{{\mathcal{F}}}
\newcommand{\T}{{\mathcal{T}}}

\begin{document}
\begin{abstract}
	Using ideas from Shelah's recent proof that a completely separable maximal almost disjoint family exists when $\c < {\aleph}_{\omega}$, we construct a weakly tight family under the hypothesis $\s \leq \b < {\aleph}_{\omega}$. The case when $\s < \b$ is handled in $\ZFC$ and does not require $\b < {\aleph}_{\omega}$, while an additional PCF type hypothesis, which holds when $\b < {\aleph}_{\omega}$ is used to treat the case $\s = \b$. The notion of a weakly tight family is a natural weakening of the well studied notion of a Cohen indestructible maximal almost disjoint family. It was introduced by Hru{\v{s}}{\'a}k and Garc{\'{\i}}a Ferreira~\cite{Hr1}, who applied it to the Kat\'etov order on almost disjoint families.
\end{abstract}
\maketitle
\section{Introduction} \label{sec:intro}
Recall that two infinite subsets $a$ and $b$ of $\omega$ are said to be \emph{almost disjoint or a.d.\@} if $a \cap b$ is finite. We say that a family $\A \subset \cube$ is \emph{almost disjoint or a.d.\@} if its elements are pairwise a.d. A \emph{Maximal Almost Disjoint or MAD} family is an infinite a.d.\ family $\A \subset \cube$ such that $\forall b \in \cube \exists a \in \A \[\lc a \cap b \rc = \omega\]$.

	MAD families have been intensively studied in set theory. They have several applications in set theory as well as general topology. For instance, the technique of almost disjoint coding has been used in forcing theory (see \cite{L}) and MAD families are used in the construction of the Isbell-Mr{\'o}wka space in topology (see \cite{G}). See \cite{bulletin} for a general survey of some recent results and open problems regarding MAD families.

	Particular attention has been focused on the existence and properties of MAD families with strong combinatorial properties. These combinatorial properties typically require the family to be ``maximal'' with respect to some additional criteria besides the one defining MAD families. The most well known is that of a completely separable MAD family. Recall that a MAD family $\A \subset \cube$ is said to be \emph{completely separable} if for any $b \in {\I}^{+}(\A)$, there is an $a \in \A$ with $a \subset b$. Here $\I(\A)$ denotes the ideal on $\omega$ generated by $\A$, and for any ideal $\I$ on $\omega$, ${\I}^{+} = \Pset(\omega) \setminus \I$.

	In most cases, it is unknown whether MAD families with these strong combinatorial properties can be constructed in $\ZFC$. In fact, almost all known constructions of such families use an assumption of the form $\x = \c$, where $\x$ is some appropriately chosen cardinal invariant. Here, the assumption $\p = \c$ serves as a limiting case, sufficing for virtually all known constructions of this sort. Upto now, there have only been two such constructions that have not required an assumption of the form $\x = \c$, the example of completely separable families first considered in \cite{ersh}, and that of Van Douwen families posed in \cite{Mi2} (see below for a discussion of Sacks indestructible MAD families).

	Balcar, Do{\v{c}}k{\'a}lkov{\'a}, and Simon~\cite{BS} obtained the first major results here by proving that a completely separable MAD family can be constructed from any of the assumptions $\b = \d$, $\s = {\omega}_{1}$, or $\d \leq \a$. Then Shelah~\cite{Sh:935} recently achieved a breakthrough by constructing such a family from $\c < {\aleph}_{\omega}$.

	Recall that a an a.d.\ family of total functions $\A \subset \BS$ is said to be \emph{Van Douwen} if for each $p$, an infinite partial function from $\omega$ to $\omega$, there is $f \in \A$ such that $\lc p \cap f\rc = \omega$. While it is easy to construct a Van Douwen family from an assumption like $\ae = \c$, Raghavan~\cite{vmad} showed how to get such an object just in $\ZFC$ alone.

	Another prominent example of a strong combinatorial property which has been considered for a.d.\ families is that of indestructibility.
\begin{Def}
	Let $\P$ be a notion of forcing and let $\A \subset \cube$ be a MAD family. We will say that \emph{$\A$ is $\P$-indestructible} if $\; {\forces}_{\P}{\; \A \ \text{is MAD}}$.
\end{Def}
	There is no forcing notion $\P$ adding a new real for which a $\ZFC$ construction of a $\P$-indestructible MAD family is known. A Sacks indestructible MAD family is provably the weakest such object in the sense that if $\A \subset \cube$ is a MAD family that is $\P$-indestructible for some $\P$ which adds a new real, then $\A$ is also Sacks indestructible. It is not too hard to see that if $\a < \c$, then any MAD family of size $\a$ is Sacks indestructible. However, the only known constructions of a Sacks indestructible MAD family of size $\c$ use either $\b = \c$ or $\cov{(\M)} = \c$. It remains an open problem whether such families (of any size) can be built in $\ZFC$.
Another basic example of a $\P$ that adds a real is Cohen forcing. Cohen indestructibility is closely related to another combinatorial property of MAD families first considered by Malykhin~\cite{Mal}.
\begin{Def}
	An a.d.\ family $\A \subset \cube$ is called \emph{${\aleph}_{0}$-MAD} or \emph{tight} or \emph{strongly MAD} if for every countable collection $\{{b}_{n}: n \in \omega\} \subset {\I}^{+}(\A)$, there is $a \in \A$ such that $\forall n \in \omega\[\lc {b}_{n} \cap a \rc = \omega\]$.
\end{Def} 
	It is not too difficult to see that there is a Cohen indestructible MAD family iff an ${\aleph}_{0}$-MAD family exists. The only known construction of an ${\aleph}_{0}$-MAD family (of size $\c$) uses $\b = \c$, and it is a long standing open problem whether their existence can be proved in $\ZFC$. It is shown in \cite{svmad} that the weak Freese--Nation property of $\Pset(\omega)$ (wFN($\Pset(\omega)$)), which is shown to hold in \cite{FKS} in any model gotten by adding fewer than ${\aleph}_{\omega}$ Cohen reals to a ground model satisfying $\CH$, implies that all ${\aleph}_{0}$-MAD families have size at most $\aleph_1$. In particular, it is consistent that there are no ${\aleph}_{0}$-MAD families of size $\c$. ${\aleph}_{0}$-MAD families have been studied in \cite{Ku} and \cite{Hr2}. Also, Brendle and Yatabe ~\cite{BandY} have provided combinatorial characterizations of $\P$-indestructibility for many other standard posets $\P$.

	 Hru{\v{s}}{\'a}k and Garc{\'{\i}}a Ferreira~\cite{Hr1} introduced the following natural weakening of an ${\aleph}_{0}$-MAD family.
\begin{Def} \label{def:weakly}
	An a.d.\ family $\A \subset \cube$ is called \emph{weakly tight} if for every countable collection $\{{b}_{n}: n \in \omega\} \subset {\I}^{+}(\A)$, there is $a \in \A$ such that $\existsinf n \in \omega \[\lc {b}_{n} \cap a \rc = \omega \]$. 
\end{Def}
They proved that such families are almost maximal in the Kat\'etov order on a.d.\ families. Given a.d.\ families $\A$ and $\B$, we say that $\A$ is \emph{Kat\'etov below} $\B$ and write $\A \; {\leq}_{K} \; \B$ if there is a function $f \in \BS$ such that $\forall a \in \I(\A)\[{f}^{-1}(a) \in \I(\B)\]$. They showed that if $\A$ is weakly tight, then for any other MAD family $\B$, if $\A \; {\leq}_{K} \; \B$, then there is a $c \in {\I}^{+}(\A)$ such that $\B \; {\leq}_{K}\; \{a \cap c: a \in \A \}$. It is unknown whether it is consistent to have a MAD family that is Kat\'etov maximal. Till now, the only known construction of a weakly tight family was from $\b = \c$, and that construction does not distinguish them in any way from ${\aleph}_{0}$-MAD families. 

	In this paper, we prove that weakly tight families exist when $\s < \b$, and that they also exist when $\s = \b$ provided that a certain PCF type hypothesis holds. By a PCF type hypothesis, we mean a hypothesis about $\cf{(\langle {\[\kappa\]}^{\omega}, \subset \rangle)}$ for some cardinal $\kappa$. Such hypotheses typically hold below ${\aleph}_{\omega}$. Our construction is a modification of Shelah \cite{Sh:935}, which in turn is a modification of the classic constructions of Balcar, Do{\v{c}}k{\'a}lkov{\'a}, and Simon~\cite{BS}. Shelah~\cite{Sh:935} shows that there is a completely separable MAD family in any of the following three situations -- case 1: $\s < \a$; case 2: $\s = \a$ and a certain PCF type hypothesis holds; case 3: $\a < \s$ plus a stronger PCF type assumption. Therefore, we prove the exact analogues of Shelah's cases 1 and 2 for weakly tight families, except that we compare $\s$ to $\b$ instead of $\a$. However, we cannot prove the analogue of case 3, and we conjecture that it cannot be done (see Conjecture \ref{c:noweaklytight}).

	However, our approach is somewhat different from \cite{Sh:935}. We first introduce a new cardinal invariant ${\s}_{\omega, \omega}$, and prove outright in $\ZFC$ that a weakly tight family exists if ${\s}_{\omega, \omega} \leq \b$.
\begin{Def}\label{def:somegaomega}
	${\s}_{\omega, \omega}$ is the least $\kappa$ such that there is a family $\{{e}_{\alpha}: \alpha < \kappa\} \subset \cube$ such that for any collection $\{{b}_{n}: n \in \omega \} \subset \cube$, there exists $\alpha < \kappa$ such that $\existsinf n \in \omega \[\lc {b}_{n} \cap {e}_{\alpha} \rc = \omega\]$ and $\existsinf n \in \omega \[ \lc {b}_{n} \cap {\bar{e}}_{\alpha} \rc = \omega \]$.  
\end{Def}
${\s}_{\omega, \omega}$ is a minor variation of $\s$ and it is equal to $\s$ in all models we know of (see Question \ref{q:sww=s}). An advantage of our approach is that it shows that the PCF hypothesis can be eliminated from case 2 so long as one is willing to replace $\s$ with ${\s}_{\omega, \omega}$. Indeed, our proof shows that this will also work for completely separable MAD families -- i.e.\ we can prove (in ZFC) that they exist under ${\s}_{\omega, \omega} \leq \a$. Also, it is easy to show, by the same argument as for $\s$, that ${\s}_{\omega, \omega} \leq \d$. So, as a corollary, we get in $\ZFC$ that weakly tight families exist under $\b = \d$. We don't know if $\b = \d$ yields the PCF assumption $P(\b)$ (see Definition \ref{def:uk}) used in the proof of the $\s = \b$ case. In Section \ref{sec:main2} we first show in $\ZFC$ that $\s = {\s}_{\omega, \omega}$ when $\s < \b$, thus getting case 1 as a corollary. Then we show that if $\kappa = \s = \b$ and $P(\kappa)$ holds, then a weakly tight family can be constructed. Here $P(\kappa)$ is our PCF type hypothesis, and it appears to be slightly stronger than the one used by Shelah for his case 2. $P(\kappa)$ is always true for $\kappa < {\aleph}_{\omega}$, so we get weakly tight families when $\s = {\omega}_{1}$. It is of some interest that we now get weakly tight families in two of the three cases (i.e.\ $\b = \d$ and $\s = {\omega}_{1}$) in which Balcar, Do{\v{c}}k{\'a}lkov{\'a}, and Simon~\cite{BS} had previously gotten completely separable MAD families.

It is also worth pointing out here that one cannot construct an ${\aleph}_{0}$-MAD family of size $\c$ from $\s \leq \b < {\aleph}_{\omega}$ because of the previously mentioned result that wFN($\Pset(\omega)$) implies there are no ${\aleph}_{0}$-MAD families of size $\c$. In particular, in the Cohen model there is a weakly tight MAD family of size $\c$, but no ${\aleph}_{0}$-MAD families of that size.

We now make some general remarks on the basic method. Suppose $\kappa = \s$. First each node $\eta$ of ${2}^{< \kappa}$ is labelled with a subset of $\omega$, say ${e}_{\eta}$. Each member of the a.d.\ family under construction is ``associated'' with a node, and the idea is that whenever two sets are associated with incomparable nodes, they are automatically a.d. This is ensured by specifying at each node of ${2}^{< \kappa}$ a collection of subsets of $\omega$ that are ``allowed'' to be associated with that node. Then most of the argument goes into showing that at any stage $\alpha < \c$ there is a prefect set of nodes with which ${a}_{\alpha}$ is allowed to be associated. Here ${a}_{\alpha}$ is the member of the a.d.\ family constructed at stage $\alpha$. This means that ${a}_{\alpha}$ can be associated with a node that is incomparable with ``most'' (all but fewer than $\s$) of the nodes with which some ${a}_{\beta}$ has already been associated. So ${a}_{\alpha}$ will be automatically a.d.\ from most of the previous ${a}_{\beta}$.

For constructing a completely separable MAD family, we can simply require that a set $a$ is allowed to be associated with a node $\eta$ iff for each node $\tau \subsetneq \eta$, $a$ is either almost included in ${e}_{\tau}$ or almost disjoint from ${e}_{\tau}$ depending on which way $\eta$ went at $\dom{(\tau)}$. However, this requirement is too strong for building a weakly tight family. Recall that a \emph{partitioner} of an a.d.\ family $\A$ is a set $b \in {\I}^{+}(\A)$ with the property that $\forall a \in \A \[a \; {\subset}^{\ast} \; b \vee \lc a \cap b \rc < \omega\]$. It is clear that any $\A$ that is subject to the above mentioned constraint will have an infinite pairwise disjoint family of partitioners. However, such an $\A$ must necessarily fail to be weakly tight. We deal with this using two innovations. Firstly, each member of the a.d.\ family will be associated with a countable collection of nodes, instead of one single node, and will be the union of a countable sequence of infinite subsets of $\omega$. Secondly, each such countable sequence will be associated with its own node, and the collection ${\I}_{\eta}$ of countable sequences allowable at a node $\eta$ will be defined so as to ensure almost disjointness (see Definition \ref{def:Ieta}).

We believe that these adaptations we have introduced for building a weakly tight family will be of use in getting other kinds of MAD families with few partitioners (see Conjecture \ref{c:sacks}) by helping us to replace assumptions of the form $\x = \c$ with weaker hypotheses of the form $\x \leq \y$. Eventually they should either show us how to do a $\ZFC$ construction or tell us where to look for a consistency proof.   
\section{The main construction} \label{sec:main}
In this section we give the PCF free construction of a weakly tight family.
\begin{Theorem} \label{thm:main}
If ${\s}_{\omega, \omega} \leq \b$, then there is a weakly tight family of size $\c$. In particular, such families exist if $\b = \d$, or if ${\s}_{\omega, \omega} = {\omega}_{1}$.
\end{Theorem}
Both the construction given here and the one in Section \ref{sec:main2} are very similar, and we could have presented a single, unified construction, and then derived the two results as corollaries. However, we have chosen to separate them because we feel that the construction presented in this section is the easiest one to follow, and a reader who has understood it should have no difficulty in assimilating the modifications made to it in Section \ref{sec:main2}. We first fix some notation.

	For any $e \subset \omega$, we use $\bar{e}$ to denote $\omega \setminus e$. We will also often use ${e}^{0}$ to denote $e$ and ${e}^{1}$ to denote $\bar{e}$. Next, we give the definition of ${\I}_{\eta}$, which should be thought of as the collection of sequences of sets that are allowable at $\eta$.
\begin{Def} \label{def:Ieta}
	We say that a sequence $\vec{C} = \langle {c}_{n}: n \in \omega \rangle \subset \cube$ is a \emph{sequence of columns} if for any $n \neq m$, ${c}_{n} \cap {c}_{m} = 0$. Define 
	\begin{align*}
		\C = \left\{\vec{C}: \vec{C} \ \text{is a sequence of columns}\right\}.
	\end{align*}
	Let $\kappa$ be an infinite cardinal, and let $\langle {e}_{\alpha}: \alpha < \kappa \rangle \subset \cube$. For an $\eta \in {2}^{\leq \kappa}$, we define
	\begin{align*}
		{\I}_{\eta}(\langle {e}_{\alpha}: \alpha < \kappa \rangle) = \left\{\vec{C} \in \C: \forall \beta < \dom{(\eta)} \forallbutfin n \in \omega \[\vec{C}(n) \subset {e}^{\eta(\beta)}_{\beta}\] \right\}.
	\end{align*}
We will often omit the $\langle {e}_{\alpha}: \alpha < \kappa \rangle$ because it will be clear from the context.
\end{Def} 
\begin{Lemma} \label{lem:positivesplitting}
Let $\langle {e}_{\alpha}: \alpha < \kappa \rangle$ witness $\kappa = {\s}_{\omega, \omega}$. Let $\A \subset \cube$ be any a.d.\ family. Then for each $b \in {\I}^{+}(\A)$, there is an $\alpha < \kappa$ such that $b \cap {e}_{\alpha} \in {\I}^{+}(\A)$ and $b \cap {\bar{e}}_{\alpha} \in {\I}^{+}(\A)$.  
\end{Lemma}
\begin{proof}
	There are two cases to consider. Suppose first that there are only finitely many $a \in \A$ with $\lc b \cap a\rc = \omega$. Then since $b \in {\I}^{+}(\A)$, there is a $c \in {\[b\]}^{\omega}$ which is a.d.\ from every member of $\A$. Now, choose $\alpha < \kappa$ such that $\lc c \cap {e}_{\alpha}\rc = \lc c \cap {\bar{e}}_{\alpha} \rc = \omega$. It is clear that this $\alpha$ is as required.

	Next, suppose that there is an infinite collection $\{{a}_{n}: n \in \omega\} \subset \A$ with $\lc b \cap {a}_{n}\rc = \omega$ for each $n \in \omega$. Put ${c}_{n} =  b \cap {a}_{n}$ and choose $\alpha < \kappa$ such that $\existsinf n \in \omega \[ \lc {c}_{n} \cap {e}_{\alpha} \rc = \omega\]$ and $\existsinf n \in \omega \[\lc {c}_{n} \cap {\bar{e}}_{\alpha} \rc = \omega\]$. Now, both $b \cap {e}_{\alpha}$ and $b \cap {\bar{e}}_{\alpha}$ are in ${\I}^{+}(\A)$ because they both have infinite intersection with infinitely many members of $\A$. 
\end{proof}
Fix a sequence $\langle {e}_{\alpha}: \alpha < \kappa \rangle$ witnessing $\kappa = {\s}_{\omega, \omega}$. We will construct an increasing sequence of subtrees of ${2}^{< \kappa}$ by induction on $\c$. The weakly tight family $\A \subset \cube$ will be constructed along with these subtrees. At a stage $\alpha < \c$, we are given an increasing sequence $\langle {\T}_{\beta}: \beta < \alpha \rangle$ of subtrees of ${2}^{< \kappa}$ such that $\lc {\T}_{\beta} \rc \leq \lc \beta \rc + \omega$ for each $\beta < \alpha$, as well as an almost disjoint family $\{{a}_{\beta}: \beta < \alpha\}$. Thus ${\T}^{\alpha} = {\bigcup}_{\beta < \alpha}{{\T}_{\beta}}$ is a subtree of ${2}^{< \kappa}$ with $\lc {\T}^{\alpha} \rc < \c$. Now, we ensure that for each $\beta < \alpha$, ${a}_{\beta} = {\bigcup}_{n \in \omega}{{d}^{\beta}_{n}}$, where ${\vec{D}}^{\beta} = \langle {d}^{\beta}_{n}: n \in \omega\rangle$ is a sequence of columns. Moreover, to each ${a}_{\beta}$ and each ${d}^{\beta}_{n}$, we associate nodes $\eta({a}_{\beta}) \in {\T}_{\beta}$ and $\eta({d}^{\beta}_{n}) \in {\T}_{\beta}$ in such a way that the following conditions are satisfied:
    \begin{align*}
	& {\vec{D}}^{\beta} \in {\I}_{\eta({a}_{\beta})}  \tag{${\dagger}_{{a}_{\beta}}$}\\ 
   	& \forall \gamma < \dom{(\eta({d}^{\beta}_{n}))} \[{d}^{\beta}_{n} {\subset}^{\ast} {e}^{\eta({d}^{\beta}_{n})(\gamma)}_{\gamma}\] \tag{${\dagger}_{{d}^{\beta}_{n}}$}
    \end{align*}
It will also be important that $\eta({a}_{\beta}) \neq \eta({a}_{\gamma})$ for all $\gamma < \beta < \alpha$, that $\eta({d}^{\beta}_{n}) \neq \eta({d}^{\gamma}_{m})$ for all $\langle \beta, n \rangle \neq \langle \gamma, m \rangle$ where $\beta, \gamma < \alpha$, and $n,m \in \omega$, and also that $\eta({a}_{\beta}) \neq \eta({d}^{\gamma}_{m})$ for all $\beta, \gamma < \alpha$, and $m \in \omega$.
The next lemma says that at each stage $\alpha < \c$, it is not the case that $\{{a}_{\beta}: \beta < \alpha\}$ is already a MAD family ``somewhere'' -- i.e.\ there is no positive set on which this family is already MAD. Having this be the case is, of course, essential if we are to meet all our $\c$ many requirements. This lemma is already sufficient for constructing a completely separable MAD family from ${\s}_{\omega, \omega} \leq \b$. For a weakly tight family, we need an analogue of this for sequences of columns (Lemma \ref{lem:Ietasplitting}).
\begin{Lemma} \label{lem:existsperfecttree}
Let $b \in {\I}^{+}(\{{a}_{\beta}: \beta < \alpha\})$. Let ${\T}^{\alpha} \subset \T$ be a subtree of ${2}^{< \kappa}$ with $\lc \T \rc < \c$. There is a $c \in {\[b\]}^{\omega}$ which is a.d.\ from ${a}_{\beta}$ for every $\beta < \alpha$, and a $\tau \in \left( {2}^{< \kappa} \right) \setminus \T$ such that $\forall \delta < \dom{(\tau)}\[c \; {\subset}^{\ast} \; {e}^{\tau(\delta)}_{\delta}\]$ 
\end{Lemma}
\begin{proof}
	Put ${\A}_{\alpha} = \{{a}_{\beta}: \beta < \alpha\}$. Build a prefect subtree $\PP = \{{\sigma}_{s}: s \in {2}^{< \omega}\}$ of ${2}^{< \kappa}$ as follows. To obtain ${\sigma}_{0}$ apply Lemma \ref{lem:positivesplitting} to find the least ${\gamma}_{0} < \kappa$ such that both $b \cap{e}_{{\gamma}_{0}}$ and $b \cap {\bar{e}}_{{\gamma}_{0}}$ are in ${\I}^{+}{({\A}_{\alpha})}$. It follows that for each $\delta < {\gamma}_{0}$, there is a unique $i \in 2$ such that $b \cap {e}^{i}_{\delta} \in {\I}^{+}{({\A}_{\alpha})}$. Define ${\sigma}_{0}: {\gamma}_{0} \rightarrow 2$ by ${\sigma}_{0}(\delta) = i$ iff $b \cap {e}^{i}_{\delta} \in {\I}^{+}{({\A}_{\alpha})}$ for each $\delta < {\gamma}_{0}$. Now, suppose $\{{\sigma}_{s}: s \in {2}^{\leq n}\} \subset {2}^{< \kappa}$ and $\{{\gamma}_{s}: s \in {2}^{\leq n}\} \subset \kappa$ have been constructed. For $s \in {2}^{\leq n + 1}$, define $e(s)$ as follows. Let $e(0)$ denote $\omega$. Given $e(s)$ for $s \in {2}^{\leq n}$, let $e({s}^{\frown}{\langle i \rangle}) = e(s) \cap {e}^{i}_{{\gamma}_{s}}$. Note that $e({s}^{\frown}{\langle i \rangle}) \subset e(s)$. Now, assume that the following properties hold:
	\begin{enumerate}
		\item
			$\forall s \in {2}^{\leq n}\[\dom{({\sigma}_{s})} =  {\gamma}_{s}\]$ and $\forall s \in {2}^{< n}\[{\sigma}_{{s}^{\frown}{\langle i \rangle}} \supset {{\sigma}_{s}}^{\frown}{\langle i \rangle}\]$
		\item
			$\forall s \in {2}^{\leq n}\[b \cap e({s}^{\frown}{\langle 0 \rangle}) \in {\I}^{+}{({\A}_{\alpha})} \ \text{and} \ b \cap e({s}^{\frown}{\langle 1 \rangle}) \in {\I}^{+}{({\A}_{\alpha})}\]$
		\item
			$\forall s \in {2}^{\leq n} \forall \delta < {\gamma}_{s}\[{\sigma}_{s}(\delta) = i \ \ \text{iff} \ b \cap e(s) \cap {e}^{i}_{\delta} \in {\I}^{+}{({\A}_{\alpha})}\]$.	
	\end{enumerate} 
Note that condition (3) entails that for each $s \in {2}^{\leq n}$ and $\delta < {\gamma}_{s}$, $b \cap e(s) \cap {e}^{1 - {\sigma}_{s}{(\delta)}}_{\delta} \notin {\I}^{+}{({\A}_{\alpha})}$. Now, given $s \in {2}^{\leq n}$ and $i \in 2$, apply Lemma \ref{lem:positivesplitting} to find the least $\gamma < \kappa$ such that both $b \cap e({s}^{\frown}{\langle i \rangle}) \cap {e}^{0}_{\gamma}$ and $b \cap e({s}^{\frown}{\langle i \rangle}) \cap {e}^{1}_{\gamma}$ are in ${\I}^{+}{({\A}_{\alpha})}$. Again, for each $\delta < {\gamma}$, there is a unique $j \in 2$ such that $b \cap e({s}^{\frown}{\langle i \rangle}) \cap {e}^{j}_{\delta} \in {\I}^{+}{({\A}_{\alpha})}$. Moreover, by (3), for each $\delta < {\gamma}_{s}$, $b \cap e({s}^{\frown}{\langle i \rangle}) \cap {e}^{1 - {\sigma}_{s}{(\delta)}}_{\delta} \notin {\I}^{+}{({\A}_{\alpha})}$. Also, $b \cap e({s}^{\frown}{\langle i \rangle}) \cap {e}^{1 - i}_{{\gamma}_{s}} = 0$. Therefore, $\gamma > {\gamma}_{s}$. Thus if we define ${\gamma}_{{s}^{\frown}{\langle i \rangle}} = \gamma$, and ${\sigma}_{{s}^{\frown}{\langle i \rangle}}: {\gamma}_{{s}^{\frown}{\langle i \rangle}} \rightarrow 2$ by ${\sigma}_{{s}^{\frown}{\langle i \rangle}}(\delta) = j$ iff $b \cap e({s}^{\frown}{\langle i \rangle}) \cap {e}^{j}_{\delta} \in {\I}^{+}{({\A}_{\alpha})}$ for each $\delta < {\gamma}_{{s}^{\frown}{\langle i \rangle}}$, then ${\sigma}_{{s}^{\frown}{\langle i \rangle}} \supset {{\sigma}_{s}}^{\frown}{\langle i \rangle}$, and conditions (1)--(3) hold.

	Now, since $\lc \T \rc < \c$, there is $f \in {2}^{\omega}$ such that $\tau = {\bigcup}_{n \in \omega}{{\sigma}_{f \restrict n}} \notin \T$. Notice that $b \; \cap \; e(f \restrict 0) \supset b \; \cap \; e(f \restrict 1) \supset \dotsb$  is a decreasing sequence of sets in ${\I}^{+}{({\A}_{\alpha})}$. Therefore, we may choose ${b}_{0} \in {\[b\]}^{\omega}$ such that ${b}_{0} \in {\I}^{+}{({\A}_{\alpha})}$, and ${b}_{0} \; {\subset}^{\ast} \; b \; \cap \; e(f \restrict n)$ for each $n \in \omega$. We claim that for all $\delta < \gamma = \sup\{{\gamma}_{f \restrict n}: n \in \omega\}$, ${b}_{0} \cap {e}^{1 - \tau(\delta)}_{\delta} \notin {\I}^{+}{({\A}_{\alpha})}$. Indeed, if $\delta < \gamma$, then $\delta < {\gamma}_{f \restrict n}$ for some $n \in \omega$, and so by (3), $b \cap e(f \restrict n) \cap {e}^{1 - \tau(\delta)}_{\delta} \notin {\I}^{+}{({\A}_{\alpha})}$. And since ${b}_{0} {\subset}^{\ast} b \cap e(f \restrict n)$, the claim follows. Therefore, for each $\delta < \gamma$, there is a finite set ${\F}_{\delta} \subset {\A}_{\alpha}$ such that ${b}_{0} \cap {e}^{1 - \tau(\delta)}_{\delta} {\subset}^{\ast} \bigcup{{\F}_{\delta}}$. Put $\F = {\bigcup}_{\delta < \gamma}{{\F}_{\delta}}$, and observe that $\lc \F \rc \leq \lc \gamma \rc$. Observe also that since ${\gamma}_{f \restrict n} < {\gamma}_{f \restrict (n + 1)}$, $\gamma$ is a limit ordinal and that $\cf(\gamma) = \omega$. Next, put $\GG = \left\{{a}_{\beta}: \[\beta < \alpha\] \wedge \[\eta({a}_{\beta}) \subset \tau \vee \exists n \in \omega \[\eta({d}^{\beta}_{n}) \subset \tau\]\]\right\}$, and note that $\lc \GG \rc \leq \lc \gamma \rc$, and that $\lc \F \cup \GG \rc \leq \lc \gamma \rc$. Now, if there exists a set $c \in {\[{b}_{0}\]}^{\omega}$ which is a.d.\ from every $a \in \F \cup \GG$, then for each $\delta < \gamma$, $c \cap {e}^{1 - \tau(\delta)}_{\delta}$ is finite, and hence $c \; {\subset}^{\ast} \; {e}^{\tau(\delta)}_{\delta}$. We claim that such a $c$ must be a.d.\ from every ${a}_{\beta}\in {\A}_{\alpha}$. Fix ${a}_{\beta} \in {\A}_{\alpha}$, and recall that ${a}_{\beta} = {\bigcup}_{n \in \omega}{{d}^{\beta}_{n}}$, where ${\vec{D}}^{\beta} = \langle {d}^{\beta}_{n}: n \in \omega \rangle$ is a sequence of columns. Since $\tau \notin \T$, $\eta({a}_{\beta}) \not\supset \tau$, and there is no $n \in \omega$ such that $\eta({d}^{\beta}_{n}) \supset \tau$. If either $\eta({a}_{\beta}) \subset \tau$, or there exists an $n \in \omega$ such that $\eta({d}^{\beta}_{n}) \subset\tau$, then ${a}_{\beta} \in \GG$, and $c \cap {a}_{\beta}$ is finite. So suppose that there is a $\delta < \min{\{\gamma, \dom(\eta{({a}_{\beta})})\}}$ such that $\tau(\delta) \neq \eta({a}_{\beta})(\delta)$, and also that for each $n \in \omega$, there is a ${\delta}_{n} < \min{\{\gamma, \dom{(\eta({d}^{\beta}_{n}))}\}}$ such that $\tau({\delta}_{n}) \neq \eta({d}^{\beta}_{n})({\delta}_{n})$. Since ${\vec{D}}^{\beta} \in {\I}_{\eta({a}_{\beta})}$ by $({\dagger}_{{a}_{\beta}})$, there is a $k \in \omega$ so that $\forall n \geq k\[{d}^{\beta}_{n} \subset {e}^{\eta({a}_{\beta})(\delta)}_{\delta}\]$, and $c \; {\subset}^{\ast} {e}^{1 - \eta({a}_{\beta})(\delta)}_{\delta}$. Therefore, ${\bigcup}_{n \geq k}{{d}^{\beta}_{n}} \subset {e}^{\eta({a}_{\beta})(\delta)}_{\delta}$, and so $c \cap \left( {\bigcup}_{n \geq k}{{d}^{\beta}_{n}} \right) \subset c \cap {e}^{\eta({a}_{\beta})(\delta)}_{\delta}$, which is finite. Thus,
	\begin{align*}
		c \cap {a}_{\beta} \; {=}^{\ast} \; c \cap \left( {\bigcup}_{n < k}{{d}^{\beta}_{n}}\right)
	\end{align*}
and so it suffices to show that $c \cap {d}^{\beta}_{n}$ is finite for each $n < k$. But for each such $n$, $c \; {\subset}^{\ast} \; {e}^{\tau({\delta}_{n})}_{{\delta}_{n}}$, while ${d}^{\beta}_{n} \; {\subset}^{\ast} \; {e}^{1 - \tau({\delta}_{n})}_{{\delta}_{n}}$ because of $({\dagger}_{{d}^{\beta}_{n}})$, giving us the desired conclusion.

	We next argue that there must be a $c \in {\[{b}_{0}\]}^{\omega}$ which is a.d.\ from every $a \in \F \cup \GG$. There are two cases to consider here. First suppose that $\cf({\s}_{\omega, \omega}) \neq \omega$. In this case, $\gamma < {\s}_{\omega, \omega} \leq \b \leq \a$, and so $\lc \F \cup \GG \rc < \a$. Since ${b}_{0} \in {\I}^{+}{({\A}_{\alpha})}$, there is a $c$ as required. Also, since $\dom{(\tau)} = \gamma$, we have that $\tau \in \left( {2}^{< \kappa} \right) \setminus \T$, which is as required.

	Next, suppose that $\cf({\s}_{\omega, \omega}) = \omega$. Then $\gamma$ could equal ${\s}_{\omega, \omega}$ \emph{a priori}. However, we claim that this cannot happen. To see this, note that since $\b$ is regular, in this case, we have that ${\s}_{\omega, \omega} < \b$, and so $\lc \F \cup \GG \rc \leq {\s}_{\omega, \omega} < \b \leq \a$. So again, since ${b}_{0} \in {\I}^{+}{({\A}_{\alpha})}$, there is $c \in {\[{b}_{0}\]}^{\omega}$ which is a.d.\ from every $a \in \F \cup \GG$. Now, we have argued above that for any such $c$, $\forall \delta < \gamma \[c \; {\subset}^{\ast} \; {e}^{\tau(\delta)}_{\delta} \]$. So if $\gamma = {\s}_{\omega, \omega}$, then there would be no $\delta < {\s}_{\omega, \omega}$ such that ${e}_{\delta}$ split $c$, contradicting the definition of ${\s}_{\omega, \omega}$. Therefore, $\gamma < {\s}_{\omega, \omega} = \kappa$, and again $\tau \in \left( {2}^{< \kappa} \right) \setminus \T$, as needed.
\end{proof}
\begin{Def} \label{def:refining}
	We say that a sequence of columns $\vec{D}$ \emph{refines} another such sequence $\vec{C}$, and write $\vec{D} \prec \vec{C}$, if there is a sequence $\langle {k}_{n}: n \in \omega \rangle \subset \omega$ such that $\forall n \in \omega \[{k}_{n + 1} > {k}_{n} \ \text{and} \ \vec{D}(n) \subset \vec{C}({k}_{n})\]$. Given $e \in \cube$ and $i \in \omega$, $e(i)$ denotes the $i$th element of $e$. Given a sequence of columns $\vec{C}$, and $e \in \cube$, $\vec{C} \restrict e$ is the sequence of columns defined by $\left( \vec{C} \restrict e \right)(n) = \vec{C}(e(n))$ for each $n \in \omega$. It is clear that $\prec$ is a transitive relation, and that $\forall \vec{C} \in \C \forall e \in \cube \[ \vec{C} \restrict e \prec \vec{C} \]$.
\end{Def}
This next lemma is the analogue of Lemma \ref{lem:existsperfecttree} for sequences of columns. It is here that comparing $\s$ to $\b$ rather than to $\a$ becomes important.
\begin{Lemma} \label{lem:Ietasplitting}
	Let ${\A}_{\alpha} = \{{a}_{\beta}: \beta < \alpha\}$. Suppose that $\vec{C}$ is a sequence of columns such that for each $n \in \omega$, $\vec{C}(n)$ is a.d.\ from every member of ${\A}_{\alpha}$. There is an $\eta \in {2}^{< \kappa}$ and a $\vec{D} \in {\I}_{\eta}$ such that 
\begin{enumerate}
	\item	
		$\vec{D} \prec \vec{C}$
	\item
		$\existsinf n \in \omega \[\lc \vec{D}(n) \cap {e}^{0}_{\dom{(\eta)}}\rc = \omega\]$
	\item
		$\existsinf n \in \omega \[\lc \vec{D}(n) \cap {e}^{1}_{\dom{(\eta)}} \rc = \omega\]$.
\end{enumerate}
\end{Lemma}
\begin{proof}
	By definition of ${\s}_{\omega, \omega}$, there is a $\gamma < \kappa$ such that $\existsinf n \in \omega \[\lc \vec{C}(n) \cap {e}^{0}_{\gamma} \rc = \omega\]$ and $\existsinf n \in \omega \[\lc \vec{C}(n) \cap {e}^{1}_{\gamma}\rc = \omega\]$. Choose the least such $\gamma$. So for each $\delta < \gamma$, there is a unique $j \in 2$ such that $\existsinf n \in \omega \[\lc \vec{C}(n) \cap {e}^{j}_{\delta} \rc = \omega\]$. Define $\eta: \gamma \rightarrow 2$ by $\eta(\delta) = j$ iff $\existsinf n \in \omega \[\lc \vec{C}(n) \cap {e}^{j}_{\delta} \rc = \omega\]$ for all $\delta < \gamma$. To get $\vec{D}$, note that for each $\delta < \gamma$, there is a ${k}_{\delta} \in \omega$ such that $\forall n \geq {k}_{\delta}\[\lc \vec{C}(n) \cap {e}^{1 - \eta(\delta)}_{\delta} \rc < \omega \]$. So we can define a function ${f}_{\delta}: \omega \rightarrow \omega$ by ${f}_{\delta}(n) = \max{\left(\vec{C}(n) \cap {e}^{1 - \eta(\delta)}_{\delta}\right)}$ for each $n \geq {k}_{\delta}$, and ${f}_{\delta}(n) = 0$, for each $n < {k}_{\delta}$. Now, since $\gamma < {\s}_{\omega, \omega} \leq \b$, find a function $f \in \BS$ such that for each $\delta < \gamma$, $\forallbutfin n \in \omega\[f(n) > {f}_{\delta}(n)\]$. Now, put $\vec{D}(n) = \vec{C}(n) \setminus f(n)$. It is clear that $\vec{D} \prec \vec{C}$. Also, since $\vec{D}(n) \; {=}^{\ast} \; \vec{C}(n)$, (2) and (3) are satisfied by the choice of $\gamma$. Finally to see that $\vec{D} \in {\I}_{\eta}$, fix $\delta < \gamma$. There is an $m \in \omega$ such that $\forall n \geq m\[f(n) > {f}_{\delta}(n)\]$. Now, suppose that $n \geq \max\{{k}_{\delta}, m\}$. Then if $l \in \vec{D}(n)$, then $l \in \vec{C}(n)$ and $l > \max{\left(\vec{C}(n) \cap {e}^{1 - \eta(\delta)}_{\delta}\right)}$, whence $l \in {e}^{\eta(\delta)}_{\delta}$. Thus we have shown that $\forallbutfin n \in \omega\[\vec{D}(n) \subset {e}^{\eta(\delta)}_{\delta}\]$.     
\end{proof}
The next lemma is easy, but plays a crucial role in the construction, and depends a lot on having the right definition of ${\I}_{\eta}$. It is a sticking point in further applications of this technique that needs to be resolved each time by finding a definition of ${\I}_{\eta}$ that is appropriate for the specific type of a.d.\ family being sought.
\begin{Lemma} \label{lem:prec}
Suppose $\langle{\sigma}_{n}: n \in \omega\rangle \subset {2}^{< \kappa}$, $\langle {\gamma}_{n} : n \in \omega\rangle \subset \kappa$, and $\langle {\vec{C}}_{n}: n \in \omega \rangle \subset \C$ are sequences such
\begin{enumerate}
	\item
		$\forall n \in \omega \[\dom{({\sigma}_{n}) = {\gamma}_{n} \ \text{and} \ {\gamma}_{n + 1} > {\gamma}_{n}} \ \text{and} \ {\sigma}_{n + 1} \supset {\sigma}_{n}\]$
	\item
		$\forall n \in \omega\[{\vec{C}}_{n} \in {\I}_{{\sigma}_{n}} \ \text{and} \ {\vec{C}}_{n + 1} \prec {\vec{C}}_{n}\]$.
\end{enumerate}
Then there is a sequence of columns $\vec{D} \in {\I}_{\sigma}$, where $\sigma = {\bigcup}_{n \in \omega}{{\sigma}_{n}}$, such that $\forall n \in \omega \[\left( \vec{D} \restrict [n , \omega) \right) \prec {\vec{C}}_{n} \]$. 
\end{Lemma}
\begin{proof}
	Simply define a sequence of columns $\vec{D}$ by $\vec{D}(n) = {\vec{C}}_{n}(n)$. Note that $\vec{D}$ is indeed a sequence of columns because if $n < l$, then since ${\vec{C}}_{l} \prec {\vec{C}}_{n}$, $\vec{D}(l) = {\vec{C}}_{l}(l) \subset {\vec{C}}_{n}({k}_{l})$ for some ${k}_{l} \geq l > n$. Therefore, ${\vec{C}}_{n}(n) \cap {\vec{C}}_{n}({k}_{l}) = 0$, and so $\vec{D}(n) \cap \vec{D}(l) = 0$. Put $\gamma = \sup\{{\gamma}_{n}: n \in \omega\}$, and note that $\gamma \leq \kappa$ is a limit ordinal with $\cf(\gamma) = \omega$. Now, we claim that for each $\delta < \gamma$, $\forallbutfin n \in \omega \[\vec{D}(n) \subset {e}^{\sigma(\delta)}_{\delta}\]$. Indeed, given $\delta < \gamma$, fix $i \in \omega$ such that $\delta < {\gamma}_{i}$. Now, there is an $m \in \omega$ such that $\forall n \geq m \[{\vec{C}}_{i}(n) \subset {e}^{\sigma(\delta)}_{\delta}\]$. Put $l = \max\{m, i\}$. Suppose $n \geq l$. Then since ${\vec{C}}_{n} \prec {\vec{C}}_{i}$, there is a ${k}_{n} \geq n$ such that $\vec{D}(n) = {\vec{C}}_{n}(n) \subset {\vec{C}}_{i}({k}_{n}) \subset {e}^{\sigma(\delta)}_{\delta}$. It is also clear that ${\vec{D}} \restrict \left[n, \omega \right) \prec {\vec{C}}_{n}$ holds for each $n \in \omega$.   
\end{proof}
\begin{proof}[Proof of Theorem \ref{thm:main}]
	The argument will be similar in structure to the proof of Lemma \ref{lem:existsperfecttree}. Suppose that at stage $\alpha < \c$, we are given a collection $\{{b}_{n}: n \in \omega\} \subset \cube$ such that for each $n \in \omega$, ${b}_{n} \in {\I}^{+}{({\A}_{\alpha})}$, where ${\A}_{\alpha} = \{{a}_{\beta}: \beta < \alpha \}$. We want to find an ${a}_{\alpha}$ which is a.d.\ from ${\A}_{\alpha}$ with the property that $\lc {a}_{\alpha} \cap {b}_{n} \rc = \omega$ for infinitely many $n \in \omega$. Moreover, we want to enlarge ${\T}^{\alpha}$ to a bigger subtree, ${\T}_{\alpha}$, of ${2}^{< \kappa}$, as well as find a sequence of columns ${\vec{D}}^{\alpha}$, and nodes $\eta({a}_{\alpha})$ and $\eta({\vec{D}}^{\alpha}(n))$ in ${\T}_{\alpha}$ in such way that ${a}_{\alpha} = {\bigcup}_{n \in \omega}{{\vec{D}}^{\alpha}(n)}$ and $({\dagger}_{{a}_{\alpha}})$ and $({\dagger}_{{\vec{D}}^{\alpha}(n)})$ hold.

	First find ${c}_{n} \in {\[{b}_{n}\]}^{\omega}$ and nodes ${\tau}_{n} \in {2}^{< \kappa}$ as follows. Given $\{{c}_{i}: i < n \}$ and $\{{\tau}_{i}: i < n\}$, apply Lemma \ref{lem:existsperfecttree} with ${b}_{n}$ as $b$ and ${\T}^{\alpha} \cup \{{\tau}_{i} \restrict \delta: i < n \wedge \delta \leq \dom{({\tau}_{i})}\}$ as $\T$ to find ${c}_{n} \in {\[{b}_{n}\]}^{\omega}$ which is a.d.\ from every $a \in {\A}_{\alpha}$ and a node ${\tau}_{n} \in \left({2}^{< \kappa}\right) \setminus \T$ such that 
\begin{align*}
	\forall \delta < \dom{({\tau}_{n})}\[{c}_{n} \; {\subset}^{\ast} \; {e}^{{\tau}_{n}(\delta)}_{\delta}\]. \tag{${\ast}_{1}$}
\end{align*}  
We may also assume, by shrinking them further if necessary, that ${c}_{n} \cap {c}_{m} = 0$ for all $n \neq m$. Now, construct a perfect subtree $\PP = \{{\sigma}_{s}: s \in {2}^{< \omega}\}$ of ${2}^{< \kappa}$ together with a collection of ordinals $\{{\gamma}_{s}: s \in {2}^{< \omega}\} \subset \kappa$, and a collection of sequences of columns $\{{\vec{C}}_{s}: s \in {2}^{< \omega}\} \subset \C$ so that the following conditions are satisfied. 
\begin{enumerate}
	\item
		$\forall s \in {2}^{< \omega} \forall i \in 2 \[ \dom{({\sigma}_{s})} = {\gamma}_{s} \wedge {\sigma}_{{s}^{\frown}{\langle i \rangle}} \supset {{\sigma}_{s}}^{\frown}{\langle i \rangle}\]$
	\item
		$\forall s \in {2}^{< \omega} \[\existsinf n \in \omega \[\lc {\vec{C}}_{s}(n) \cap {e}^{0}_{{\gamma}_{s}} \rc = \omega \] \wedge \existsinf n \in \omega \[\lc {\vec{C}}_{s}(n) \cap {e}^{1}_{{\gamma}_{s}}\rc = \omega \]\]$.
	\item
		$\forall s \in {2}^{< \omega} \forall i \in 2 \[{\vec{C}}_{s} \in {\I}_{{\sigma}_{s}} \wedge {\vec{C}}_{{s}^{\frown}{\langle i \rangle}} \; \prec \; {\vec{C}}_{s} \]$.
\end{enumerate}
To start with, define a sequence of columns ${\vec{E}}_{0}$ by ${\vec{E}}_{0}(n) = {c}_{n}$. Now, suppose that ${\vec{E}}_{s} \prec {\vec{E}}_{0}$ is given for some $s \in {2}^{< \omega}$. To obtain ${\sigma}_{s}$, apply Lemma \ref{lem:Ietasplitting} to ${\vec{E}}_{s}$ to find ${\sigma}_{s} \in {2}^{< \kappa}$ and a ${\vec{C}}_{s} \prec {\vec{E}}_{s}$ such that ${\vec{C}}_{s} \in {\I}_{{\sigma}_{s}}$, and $\existsinf n \in \omega \[\lc {\vec{C}}_{s}(n) \cap {e}^{0}_{{\gamma}_{s}} \rc  = \omega\]$ and $\existsinf n \in \omega \[\lc {\vec{C}}_{s}(n) \cap {e}^{1}_{{\gamma}_{s}} \rc  = \omega\]$, where ${\gamma}_{s} = \dom{({\sigma}_{s})}$. Now, for each $i \in 2$, let $\langle {n}^{i}_{j}: j \in \omega \rangle$ enumerate in strictly increasing order $\left\{n \in \omega: \lc {{\vec{C}}_{s}}(n) \cap {e}^{i}_{{\gamma}_{s}} \rc = \omega \right\}$, and define a sequence of columns ${\vec{E}}_{{s}^{\frown}{\langle i \rangle}}$, by ${\vec{E}}_{{s}^{\frown}{\langle i \rangle}} (j) = {{\vec{C}}_{s}}({n}^{i}_{j}) \cap {e}^{i}_{{\gamma}_{s}}$. It is clear that (2) is satisfied. (3) will be satisfied because ${\vec{E}}_{{s}^{\frown}{\langle i \rangle}} \prec {\vec{C}}_{s}$, and therefore, ${\vec{C}}_{{s}^{\frown}{\langle i \rangle}} \; \prec \; {\vec{E}}_{{s}^{\frown}{\langle i \rangle}} \; \prec \; {\vec{C}}_{s}$. To see that (1) holds, note that ${\vec{E}}_{{s}^{\frown}{\langle i \rangle}} \in {\I}_{\left({({\sigma}_{s})}^{\frown}{\langle i \rangle} \right)}$. Since ${\gamma}_{{s}^{\frown}{\langle i \rangle}} = \dom{({\sigma}_{{s}^{\frown}{\langle i \rangle}})}$ is chosen in such a way that $\existsinf n \in \omega \[\lc {\vec{C}}_{{s}^{\frown}{\langle i \rangle}} (n) \cap {e}^{0}_{{\gamma}_{({s}^{\frown}{\langle i \rangle})}} \rc = \omega \]$ and $\existsinf n \in \omega \[\lc {\vec{C}}_{{s}^{\frown}{\langle i \rangle}} (n) \cap {e}^{1}_{{\gamma}_{({s}^{\frown}{\langle i \rangle})}} \rc = \omega \]$, it follows that ${\gamma}_{{s}^{\frown}{\langle i \rangle}} > {\gamma}_{s}$. Moreover, since ${\vec{C}}_{{s}^{\frown}{\langle i \rangle}} \in {\I}_{{\sigma}_{({s}^{\frown}{\langle i \rangle})}}$, if there is a $\delta \leq {\gamma}_{s}$ such that ${({\sigma}_{s})}^{\frown}{\langle i \rangle}(\delta) \neq {\sigma}_{{s}^{\frown}{\langle i \rangle}} (\delta)$, then there would an $n \in \omega$ such that ${\vec{C}}_{{s}^{\frown}{\langle i \rangle}}(n) \subset {e}^{0}_{\delta}$ and ${\vec{C}}_{{s}^{\frown}{\langle i \rangle}}(n) \subset {e}^{1}_{\delta}$, which is impossible. Therefore, ${\sigma}_{{s}^{\frown}{\langle i \rangle}} \supset {({\sigma}_{s})}^{\frown}{\langle i \rangle}$, and so (1) is satisfied.

	Now, put $\T = {\T}^{\alpha} \cup \{{\tau}_{n} \restrict \delta: n < \omega \wedge \delta \leq \dom{({\tau}_{n})}\}$, and note $\lc  \T \rc < \c$. Therefore, there is an $f \in {2}^{\omega}$ such that $\tau = {\bigcup}_{n \in \omega}{{\sigma}_{f \restrict n}} \notin \T$. By (1)--(3), we have that $\dom{({\sigma}_{f \restrict n})} = {\gamma}_{f \restrict n}$, that ${\gamma}_{f \restrict n + 1} > {\gamma}_{f \restrict n}$, that ${\sigma}_{f \restrict n + 1} \supset {\sigma}_{f \restrict n}$, that ${\vec{C}}_{f \restrict n} \in {\I}_{({\sigma}_{f \restrict n})}$, and that ${\vec{C}}_{f \restrict n + 1} \prec {\vec{C}}_{f \restrict n}$. So the hypotheses of Lemma \ref{lem:prec} are satisfied and we can find a sequence of columns $\vec{E} \in {\I}_{\tau}$ with $\vec{E} \prec {\vec{C}}_{0} \prec {\vec{E}}_{0}$. We set 
\begin{align*}
	\eta({a}_{\alpha}) = \tau \tag{${\ast}_{2}$}.
\end{align*} 
Notice that $\dom{(\tau)} = \gamma = \sup\{{\gamma}_{f \restrict n}: n \in \omega\}$. Clearly, $\gamma \leq \kappa$ is a limit ordinal, and $\cf{(\gamma)} = \omega$. To see that $\gamma \neq \kappa$, we argue as in Lemma \ref{lem:existsperfecttree}. If $\gamma = \kappa$, then since $\vec{E} \in {\I}_{\tau}$, there is no $\delta < \kappa$ so that $\existsinf n \in \omega \[\lc \vec{E}(n) \cap {e}^{0}_{\delta} \rc = \omega\]$ and $\existsinf n \in \omega \[\lc \vec{E}(n) \cap {e}^{1}_{\delta} \rc = \omega\]$, contradicting the definition of ${\s}_{\omega, \omega}$. Thus $\gamma < \kappa$, and so $\eta({a}_{\alpha}) \in {2}^{< \kappa}$, as needed.

Next, to define ${\vec{D}}^{\alpha}$, proceed as follows. Since $\vec{E} \prec {\vec{E}}_{0}$, $\vec{E}(n)$ is a.d.\ from ${\A}_{\alpha}$ for each $n \in \omega$. For each $\delta < \gamma$, either if there exists $\beta < \alpha$ such that $\eta({a}_{\beta}) = \tau \restrict \delta$, or if there exists a $\beta < \alpha$ and $m \in \omega$ with $\eta({d}^{\beta}_{m}) = \tau \restrict \delta$, we define a function ${f}_{\delta} \in \BS$ as follows. Given $n \in \omega$, we set ${f}_{\delta}(n) = \max{(\vec{E}(n) \cap {a}_{\beta})}$, where $\beta$, assuming it exists, is the unique $\beta < \alpha$ such that either $\eta({a}_{\beta}) = \tau \restrict \delta$ or $\eta({d}^{\beta}_{m}) = \tau \restrict \delta$ for some $m \in \omega$. Notice that since $\gamma < {\s}_{\omega, \omega} \leq \b$, we can find a function $f \in \BS$ such that
\begin{align*}
	\forall \delta < \gamma \[ \[\exists \beta < \alpha \[\eta({a}_{\beta}) = \tau \restrict \delta\] \vee \exists \beta < \alpha \exists m \in \omega \[\eta({d}^{\beta}_{m}) = \tau \restrict \delta\] \] \implies {f}_{\delta} \; {<}^{\ast} \; f\].
\end{align*}     
Now, define ${\vec{D}}^{\alpha}$ by ${\vec{D}}^{\alpha}(n) = \vec{E}(n) \setminus f(n)$ for each $n \in \omega$. It is clear that ${\vec{D}}^{\alpha} \prec \vec{E}$, and therefore, ${\vec{D}}^{\alpha} \in {\I}_{\tau}$. So $({\dagger}_{{a}_{\alpha}})$ is satisfied. Next, we put
\begin{align*}
	{a}_{\alpha} = {\bigcup}_{n \in \omega}{{\vec{D}}^{\alpha}(n)} \tag{${\ast}_{3}$}
\end{align*}
Next, suppose that the relation ${\vec{D}}^{\alpha} \prec {\vec{E}}_{0}$ is witnessed by the sequence $\langle {k}_{n}: n \in \omega \rangle$. Notice that for each $n \in \omega$, ${\vec{D}}^{\alpha}(n) \in {\[{b}_{{k}_{n}}\]}^{\omega}$, and hence that $\lc {a}_{\alpha} \cap {b}_{{k}_{n}}\rc = \omega$. Now, for each $n \in \omega$, we set
\begin{align*}
	\eta({\vec{D}}^{\alpha}(n)) = {\tau}_{{k}_{n}}.
\end{align*} 
By (${\ast}_{1}$), we have that for each $n \in \omega$, $\forall \delta < \dom{(\eta({\vec{D}}^{\alpha}(n)))} \[{\vec{D}}^{\alpha}(n) \; {\subset}^{\ast} \; {e}^{\eta({\vec{D}}^{\alpha}(n))(\delta)}_{\delta} \]$, hence $({\dagger}_{{\vec{D}}^{\alpha}(n)})$ is satisfied. Note also, that for each $i \in \omega$, ${\tau}_{i} \notin {\T}^{\alpha}$, and therefore, for each $\beta < \alpha$ and $m \in \omega$, $\eta({\vec{D}}^{\alpha}(n)) \neq \eta({a}_{\beta})$, and $\eta({\vec{D}}^{\alpha}(n)) \neq \eta({d}^{\beta}_{m})$. Also, since ${\tau}_{i} \neq {\tau}_{j}$ whenever $i \neq j$, we have that $\eta({\vec{D}}^{\alpha}(n)) \neq \eta({\vec{D}}^{\alpha}(m))$ whenever $n \neq m$. And similarly, since $\eta({a}_{\alpha})$ is not in ${\T}^{\alpha} \cup \{{\tau}_{n} \restrict \delta: n < \omega \wedge \delta \leq \dom{({\tau}_{n})}\}$, we have that $\eta({a}_{\alpha}) \neq \eta({\vec{D}}^{\alpha}(n))$, for any $n \in \omega$, and also that for any $\beta < \alpha$ and $m \in \omega$, $\eta({a}_{\alpha}) \neq \eta({a}_{\beta})$, and $\eta({a}_{\alpha}) \neq \eta({d}^{\beta}_{m})$. Therefore, we may set
\begin{align*}
	{\T}_{\alpha} = {\T}^{\alpha} \cup \{{\tau}_{{k}_{n}} \restrict \delta: n < \omega \wedge \delta \leq \dom{({\tau}_{{k}_{n}})}\} \cup\{\tau \restrict \delta: \delta \leq \dom{(\tau)}\}. \tag{${\ast}_{4}$}
\end{align*}
	It only remains to be seen that ${a}_{\alpha} \cap {a}_{\beta}$ is finite for each $\beta < \alpha$. Fix $\beta < \alpha$. There are two cases to consider. Suppose first that either $\eta({a}_{\beta}) \subsetneq \tau$ or that there is an $m \in \omega$ so that $\eta({d}^{\beta}_{m}) \subsetneq \tau$. In this case, ${f}_{\delta}$ is defined as above, and $\exists  k \in \omega \forall n \geq k \[f(n) > {f}_{\delta}(n) = \max{(\vec{E}(n) \cap {a}_{\beta})}\]$. It follows that ${a}_{\alpha} \cap {a}_{\beta} \subset {a}_{\beta} \cap \left({\bigcup}_{n < k}{{\vec{D}}^{\alpha}(n)}\right)$, which is finite.

Now, suppose that for every $\delta < \gamma$, $\eta({a}_{\beta}) \neq \tau \restrict \delta$, and also that for every $m \in \omega$ and every $\delta < \gamma$, $\eta({d}^{\beta}_{m}) \neq \tau \restrict \delta$. Since $\tau \notin {\T}^{\alpha}$, it follows that $\tau \not\subset \eta({a}_{\beta})$, and also that for each $m \in \omega$, $\tau \not\subset \eta({d}^{\beta}_{m})$. Therefore, there is a $\delta < \min\{\gamma, \dom{(\eta({a}_{\beta}))}\}$ such that $\tau(\delta) \neq \eta({a}_{\beta})(\delta)$, as well as ${\delta}_{m} < \min\{\gamma, \dom{(\eta({d}^{\beta}_{m}))}\}$ such that $\tau({\delta}_{m}) \neq \eta({d}^{\beta}_{m})({\delta}_{m})$, for each $m \in \omega$. Hence there are ${k}_{\alpha} \in \omega$ and ${k}_{\beta} \in \omega$ such that $\forall n \geq {k}_{\beta}\[{d}^{\beta}_{n} \subset {e}^{1 - \tau(\delta)}_{{\delta}}\]$ and $\forall n \geq {k}_{\alpha}\[{\vec{D}}^{\alpha}(n) \subset {e}^{\tau(\delta)}_{{\delta}}\]$. Put $d = {\bigcup}_{n \geq {k}_{\beta}}{{d}^{\beta}_{n}}$. Notice that ${a}_{\alpha} \cap d = \left( {\bigcup}_{n < {k}_{\alpha}}{\left({\vec{D}}^{\alpha}(n) \cap d\right)} \right) \cup \left( {\bigcup}_{n \geq {k}_{\alpha}}{\left({\vec{D}}^{\alpha}(n) \cap d\right)} \right)$, and this is finite because ${\vec{D}}^{\alpha}(n) \cap d = 0$ when $n \geq {k}_{\alpha}$, and ${\vec{D}}^{\alpha}(n) \cap d$ is finite for all $n \in \omega$ since ${\vec{D}}^{\alpha}(n)$ is a.d.\ from ${a}_{\beta}$. So it suffices to show that ${a}_{\alpha} \cap \left( {\bigcup}_{n < {k}_{\beta}}{{d}^{\beta}_{n}}\right)$ is finite, and for this it is enough to show that ${a}_{\alpha} \cap {d}^{\beta}_{n}$ is finite for every $n \in \omega$. To see this, fix $n \in \omega$. By assumption, there is a $k \in \omega$ such that $\forall m \geq k \[{\vec{D}}^{\alpha}(m) \subset {e}^{\tau({\delta}_{n})}_{{\delta}_{n}}\]$, while ${d}^{\beta}_{n} \; {\subset}^{\ast} \; {e}^{1 - \tau({\delta}_{n})}_{{\delta}_{n}}$. It follows that $ {d}^{\beta}_{n} \cap {a}_{\alpha} \subset \left( {\bigcup}_{m < k}{\left( {d}^{\beta}_{n} \cap {\vec{D}}^{\alpha}(m) \right)} \right) \cup \left( {d}^{\beta}_{n} \cap {e}^{\tau({\delta}_{n})}_{{\delta}_{n}}\right)$, which is finite because ${\vec{D}}^{\alpha}(m)$ is a.d.\ from ${a}_{\beta}$, and hence from ${d}^{\beta}_{n}$.  
\end{proof}
\section{Using PCF type assumptions} \label{sec:main2}
In this section, we show that ${\s}_{\omega, \omega}$ can be replaced in Theorem \ref{thm:main} by $\s$ in the presence of a relatively weak PCF type hypothesis. This hypothesis is only needed when $\s = \b$ -- when $\s < \b$ we get a $\ZFC$ result. In fact, we are able to show that when $\s < \b$, $\s = {\s}_{\omega, \omega}$, so Theorem \ref{thm:main} can be directly applied. This gives us an exact analogue of case 1 of Shelah's construction, where he gets a completely separable MAD family from $\s < \a$ without further hypotheses.

	When $\s = \b$ we seem to need a slightly stronger hypothesis than the one used by Shelah. For his construction Shelah uses the following:
\begin{Def} \label{def:ukk}
For a cardinal $\kappa > \omega$, $U(\kappa)$ is the following principle. There is a sequence $\langle {u}_{\alpha}: \omega \leq \alpha < \kappa \rangle$ such that
	\begin{enumerate}
		\item
			${u}_{\alpha} \subset \alpha$ and $\lc {u}_{\alpha} \rc = \omega$	
		\item
			$\forall X \in {\[\kappa\]}^{\kappa} \exists \omega \leq \alpha < \kappa\[\lc {u}_{\alpha} \cap X \rc = \omega\]$.
	\end{enumerate}
\end{Def}
It is easily seen that $U(\kappa)$ holds whenever $\kappa < {\aleph}_{\omega}$, and more generally whenever $\cf{(\langle {\[\kappa\]}^{\omega}, \subset \rangle)} = \kappa$. Shelah~\cite{Sh:935} (see Section 2) showed that if $\kappa = \s = \a$ and $U(\kappa)$ holds, then there is there is a completely separable MAD family. Our result will use the principle $P(\kappa)$ given below. But we first dispose of the easy case -- i.e.\ $\s < \b$.    
\begin{Theorem} \label{thm:whenbisbig}
	If $\s < \b$, then $\s = {\s}_{\omega, \omega}$. So there is a weakly tight family of size $\c$ under $\s < \b$.
\end{Theorem}
\begin{proof}
	Let $\langle {e}_{\alpha}: \alpha < \kappa \rangle$ witness that $\kappa = \s$. Suppose $\{{b}_{n}: n \in \omega \} \subset \cube$ is a countable collection such that $\forall \alpha < \kappa \exists i \in 2 \forallbutfin n \in \omega \[{b}_{n} \; {\subset}^{\ast} \; {e}^{i}_{\alpha}\]$. By shrinking them if necessary we may assume that ${b}_{n} \cap {b}_{m} = 0$ whenever $n \neq m$. Now, for each $\alpha < \kappa$ define ${f}_{\alpha} \in \BS$ as follows. We know that there is a unique ${i}_{\alpha} \in 2$ such that there is a ${k}_{\alpha} \in \omega$ such that $\forall n \geq {k}_{\alpha} \[\lc {b}_{n} \cap {e}^{{i}_{\alpha}}_{\alpha} \rc < \omega \]$. We define ${f}_{\alpha}(n) = \max\left( {b}_{n} \cap {e}^{{i}_{\alpha}}_{\alpha} \right)$ if $n \geq {k}_{\alpha}$, and ${f}_{\alpha}(n) = 0$ if $n < {k}_{\alpha}$. As $\kappa < \b$, there is a $f \in \BS$ with $f \; {}^{\ast}{>} \; {f}_{\alpha}$ for each $\alpha < \kappa$. Now, for each $n \in \omega$, choose ${l}_{n} \in {b}_{n}$ with ${l}_{n} \geq f(n)$. Since the ${b}_{n}$ are pairwise disjoint, $c = \{{l}_{n}: n \in \omega\} \in \cube$. So by definition of $\s$, there is $\alpha < \kappa$ such that $\lc c \cap {e}^{0}_{\alpha} \rc = \lc c \cap {e}^{1}_{\alpha} \rc = \omega$. In particular, $c \cap {e}^{{i}_{\alpha}}_{\alpha}$ is infinite. But we know that there is an ${m}_{\alpha} \in \omega$ such that $\forall n \geq {m}_{\alpha}\[{f}_{\alpha}(n) < f(n)\]$. So there exists $n \geq \max\{{m}_{\alpha}, {k}_{\alpha}\}$ with ${l}_{n} \in {b}_{n} \cap {e}^{{i}_{\alpha}}_{\alpha}$. But this is a contradiction because ${l}_{n} \leq {f}_{\alpha}(n) < f(n)$.
\end{proof}
\begin{Def} \label{def:uk}
For a cardinal $\kappa > \omega$, $P(k)$ is the following principle. There is a sequence $\langle {u}_{\alpha}: \omega \leq \alpha < \kappa \rangle$ such that
	\begin{enumerate}
		\item
			${u}_{\alpha} \subset \alpha$ and $\lc {u}_{\alpha} \rc = \omega$
		\item
			$\forall \{{X}_{n}: n \in \omega \} \subset {\[\kappa\]}^{\kappa} \exists \omega \leq \alpha < \kappa \existsinf n \in \omega \[{u}_{\alpha} \cap {X}_{n} \neq 0 \]$.
	\end{enumerate} 
\end{Def}
Again, it is easy to see that $P(\kappa)$ hold whenever $\cf{(\langle {\[\kappa\]}^{\omega}, \subset \rangle)} = \kappa$. Also, it is clear that $P(\kappa) \implies U(\kappa)$. We don't know whether these principles are different. We also do not know of a model where $\kappa = \s = \b$ and $P(\kappa)$ fails. Similarly, it is not known whether $U(\kappa)$ can fail when $\kappa = \s = \a$, which is the hypothesis relevant to case 2 of Shelah's construction.
	
	The next lemma is well known and fairly standard. It allows us to assume that the order type of each ${u}_{\alpha}$ is $\omega$, and plays an important role in the construction below. We include a proof for the reader's convenience.   
\begin{Lemma}\label{lem:typeomega}
	Suppose $\b \leq \kappa$ and $P(\kappa)$ holds. Then there is a family $\langle {u}_{\alpha}: \omega \leq \alpha < \kappa \rangle$ as in Definition \ref{def:uk} with $\otp{({u}_{\alpha})} = \omega$, for each $\omega \leq \alpha < \kappa$. 
\end{Lemma}
\begin{proof}
	It is sufficient to show that for each set $y \subset \kappa$ with $\lc y \rc = \omega$ there is a family $\langle {y}_{\gamma}: \gamma < \kappa \rangle$ with
	\begin{enumerate}
		\item[(a)]
			${y}_{\gamma} \subset y$ and $\otp{({y}_{\gamma})} = \omega$
		\item[(b)]
			$\forall x \in {\[ y \]}^{\omega} \exists \gamma < \kappa \[\lc x \cap {y}_{\gamma} \rc = \omega\]$.
	\end{enumerate}
Clearly, we may assume that $\otp{(y)}$ is a limit ordinal. We will prove this claim by induction on $\otp{(y)}$. If $\otp{(y)} = \omega$, then there is nothing to do. For any $\xi < \otp{(y)}$, let $y(\xi)$ denote the $\xi$th element of $y$. If $\otp{(y)} = \delta + \omega$ for some limit $\delta$, then let $z = \{y(\xi): \xi < \delta\}$ and let $\langle {z}_{\gamma}: \gamma < \kappa \rangle$ be a family satisfying (a) and (b) with respect to $z$. Now, simply let $\langle {y}_{\gamma} : \gamma < \kappa \rangle$ be an enumeration of $\{ \{y(\delta + n): n < \omega\} \} \cup \{{z}_{\gamma}: \gamma < \kappa \}$. Next, suppose that $\otp{(y)}$ is a limit of limits. Let $\langle {\delta}_{n} : n \in \omega \rangle$ be an increasing sequence of limit ordinals converging to $\delta = \otp{(y)}$. Put ${z}_{n} = \{y(\xi): {\delta}_{n - 1} \leq \xi < {\delta}_{n}\}$, where ${\delta}_{-1}$ is taken to be 0. Let $\langle {z}^{n}_{\gamma}: \gamma < \kappa \rangle$ be a family satisfying (a) and (b) with respect to ${z}_{n}$. Now, let $\langle {f}_{\alpha}: \alpha < \b \rangle$ be a family in $\BS$ which is unbounded with respect to infinite partial functions from $\omega$ to $\omega$, and let $\{{\zeta}^{n}_{i}: i \in \omega \}$ be an enumeration of ${z}_{n}$. For each $\alpha < \b$, define a set ${y'}_{\alpha} = \{{\zeta}^{n}_{i}: i \leq {f}_{\alpha}(n)\}$. Notice that $\otp{({y'}_{\alpha})} = \omega$. Let $\langle {y}_{\gamma}: \gamma < \kappa \rangle$ enumerate $\left( {\bigcup}_{n \in \omega}{\{{z}^{n}_{\gamma}: \gamma < \kappa \}} \right) \cup \{{y'}_{\alpha}: \alpha < \b \}$. We check that this family satisfies (b) with respect to $y$. Fix $x \in {\[y\]}^{\omega}$. If $x \cap {z}_{n}$ is infinite for some $n \in \omega$, then there is a $\gamma < \kappa$ so that $\lc x \cap {z}^{n}_{\gamma} \rc =\omega$. On the other hand, if $x \cap {z}_{n}$ is finite for each $n \in \omega$, then $\existsinf n \in \omega \[x \cap {z}_{n} \neq 0 \]$. So we may pick a strictly increasing sequence $\langle {k}_{n} : n \in \omega \rangle \subset \omega$ and $\{{i}_{{k}_{n}}: n \in \omega \} \subset \omega$ such that ${\zeta}^{{k}_{n}}_{{i}_{{k}_{n}}} \in x$ for each $n \in \omega$. There is an $\alpha < \kappa$ so that $\existsinf n \in \omega \[{f}_{\alpha}({k}_{n}) \geq {i}_{{k}_{n}}\]$. Now, it is clear that $\lc x \cap {y'}_{\alpha} \rc = \omega$. 
\end{proof}
\begin{Theorem} \label{thm:main2}
Assume $\kappa = \s = \b$ and that $P(\kappa)$ holds. Then there is a weakly tight family of size $\c$. In particular, such families exist if $\s \leq \b < {\aleph}_{\omega}$, and in particular, when $\s = {\omega}_{1}$. 
\end{Theorem}
The proof of Theorem \ref{thm:main2} is very similar to the proof of Theorem \ref{thm:main}. The main difference will be that instead of using a sequence of sets $\langle {e}_{\alpha}: \alpha < \kappa \rangle$, we will construct a tree $\langle {e}_{\eta}: \eta \in {2}^{< \kappa} \rangle$. So the pair of sets $e$, $\bar{e}$ used at a node of the tree will now depend not just on the height of that node, but on all the pairs of sets that occur below that node. The idea is that along each cofinal branch $\psi$ of the tree, each countable collection of $\kappa$-sized subsets $\psi$ can be ``captured'' at some node $\eta$ that lies on $\psi$ using $P(\kappa)$. Then ${e}_{\eta}$ is chosen in such a way that for any $\{{b}_{n}: n \in \omega \} \subset \cube$, if $\{{X}_{n}: n \in \omega\}$ is the countable collection of $\kappa$-sized subsets of $\psi$ ``captured'' at $\eta$, where ${X}_{n}$ is the set of nodes on $\psi$ where ${b}_{n}$ ``hits the other side'', then $\existsinf n \in \omega \[\lc {b}_{n} \cap {e}^{1 - \psi(\dom{(\eta)})}_{\eta}\rc = \omega \]$. While the basic idea is the same as in cases 2 and 3 of Shelah's construction, there is one crucial difference here. An appropriate ${e}_{\eta}$ is chosen in Shelah's construction using a $\b$ family (quite similarly to what is done in Lemma \ref{lem:typeomega}), while we use an $\s$ family for this. If we could replace the $\s$ family in our construction by a $\b$ family, then we would also be able to prove the analogue of Shelah's case 3 -- i.e.\ we would be able to get a weakly tight family from $\b < \s < {\aleph}_{\omega}$. But we suspect that there are fundamental reasons for not being able to do this (see Conjecture \ref{c:noweaklytight}).
\begin{proof} [Proof of Theorem \ref{thm:main2}]
	First construct $\langle {e}_{\eta}: \eta \in {2}^{< \kappa} \rangle \subset \Pset(\omega)$ as follows. Let $\kappa = {\bigcup}_{\alpha < \kappa}{{S}_{\alpha}}$ be a partition of $\kappa$ so that $\lc {S}_{\alpha} \rc = \kappa$ and $\gamma \geq \alpha$ hold for each $\alpha < \kappa$ and $\gamma \in {S}_{\alpha}$. Let $\langle {u}_{\alpha}: \omega \leq \alpha < \kappa \rangle $ witness that $P(\kappa)$ holds. By Lemma \ref{lem:typeomega}, we may assume that $\otp{({u}_{\alpha})} = \omega$. Now, for each $\alpha < \kappa$, let $\langle {e}_{\gamma}: \gamma \in {S}_{\alpha} \rangle$ witness that $\kappa = \s$. We define ${e}_{\eta}$ by induction on $\dom{(\eta)}$. Assume $\dom{(\eta)} = \gamma < \kappa$, and that for each $\beta < \gamma$, ${e}_{\eta \restrict \beta} \subset \omega$ has been defined. Suppose $\gamma \in {S}_{\alpha}$. If $\alpha < \omega$, then let ${e}_{\eta} = {e}_{\gamma}$. If $\alpha \geq \omega$, we proceed as follows. Since ${u}_{\alpha}$ has order type $\omega$, enumerate it in strictly increasing order as ${u}_{\alpha} = \{{\xi}^{\alpha}_{i}: i < \omega\}$. Since $\gamma \geq \alpha > {\xi}^{\alpha}_{i}$, ${e}_{\eta \restrict {\xi}^{\alpha}_{i}}$ has already been defined. For each $i < \omega$, we put
\begin{align*}
	{c}^{\eta}_{i} = {e}^{1 - \eta({\xi}^{\alpha}_{i})}_{\eta \restrict {\xi}^{\alpha}_{i}} \cap \left( {\bigcap}_{j < i}{{e}^{\eta({\xi}^{\alpha}_{j})}_{\eta \restrict {\xi}^{\alpha}_{j}}}\right)
\end{align*} 
Notice that ${c}^{\eta}_{i} \cap {c}^{\eta}_{j} = 0$, for all $i \neq j$. We then define
\begin{align*}
	{e}_{\eta} = {\bigcup}_{i \in {e}_{\gamma}}{{c}^{\eta}_{i}} 
\end{align*}
This completes the definition of $\langle {e}_{\eta}: \eta \in {2}^{< \kappa} \rangle$. The next lemma establishes the key property of this family, which will give the analogues of Lemmas \ref{lem:positivesplitting}, \ref{lem:existsperfecttree}, and \ref{lem:Ietasplitting}.
\renewcommand{\qedsymbol}{}
\end{proof}
\vspace{-5mm}
\begin{Lemma} \label{lem:nocofinalbranch}
	Let $\{{b}_{n}: n \in \omega \} \subset \cube$, and let $\psi \in {2}^{\kappa}$. Then there is a $\gamma < \kappa$ such that $\existsinf n \in \omega \[\lc {b}_{n} \cap {e}^{1 - \psi(\gamma)}_{\psi \restrict \gamma}\rc = \omega \]$.
\end{Lemma}
\begin{proof}
	Suppose not. Fix $\psi \in {2}^{\kappa}$ such that for all $\gamma < \kappa$, $\forallbutfin n \in \omega \[ {b}_{n} \; {\subset}^{\ast} \; {e}^{\psi(\gamma)}_{\psi \restrict \gamma}\]$. For each $n \in \omega$, define
	\begin{align*}
		{X}_{n} = \left\{\gamma < \kappa: \lc {b}_{n} \cap {e}^{1 - \psi(\gamma)}_{\psi \restrict \gamma} \rc = \omega \right\}.
	\end{align*} 
We claim that $\lc {X}_{n} \rc = \kappa$. Indeed, suppose, for a contradiction, that $\lc {X}_{n} \rc < \kappa$. Put $\F = \{{e}_{\psi \restrict \gamma}: \gamma \in {X}_{n} \}$. This is a family of subsets of $\omega$ of size less than $\kappa = \s$. So we may find a $c \in {\[{b}_{n}\]}^{\omega}$ such that for each $\gamma \in {X}_{n}$, there is an $i \in 2$ so that $c \; {\subset}^{\ast} \; {e}^{i}_{\psi \restrict \gamma}$. However, $\langle {e}_{\gamma}: \gamma \in {S}_{0}\rangle$ enumerates a splitting family. So there is a $\gamma \in {S}_{0}$ so that $\lc c \cap {e}^{0}_{\psi \restrict \gamma} \rc = \lc c \cap {e}^{1}_{\psi \restrict \gamma} \rc = \omega$. In particular, $\lc {b}_{n} \cap {e}^{1 - \psi(\gamma)}_{\psi \restrict \gamma} \rc = \omega$, and so $\gamma \in {X}_{n}$. But this is a contradiction because $c \; {\subset}^{\ast} \; {e}^{i}_{\psi \restrict \gamma}$.

	Now, choose $\omega \leq \alpha < \kappa$ such that $\existsinf n \in \omega \[{u}_{\alpha} \cap {X}_{n} \neq 0\]$. We choose two strictly increasing sequences $\langle {k}_{m}: m \in \omega \rangle \subset \omega$ and $\langle {i}_{m}: m \in \omega \rangle \subset \omega$ as follows. Let ${k}_{0}$ be the least $n \in \omega$ such that ${u}_{\alpha} \cap {X}_{n} \neq 0$, and let ${i}_{0}$ be the least $i \in \omega$ such that ${\xi}^{\alpha}_{i} \in {X}_{{k}_{0}}$. Suppose that ${k}_{m}$ and ${i}_{m}$ are given to us with ${\xi}^{\alpha}_{{i}_{m}} \in {X}_{{k}_{m}}$. Put
\begin{align*}
	s = \left\{ n \in \omega: \exists i \leq {i}_{m} \lc {b}_{n} \cap {e}^{1 - \psi({\xi}^{\alpha}_{i})}_{\psi \restrict {\xi}^{\alpha}_{i}} \rc = \omega\right\}.
\end{align*}
Since for each $i \leq {i}_{m}$, $\forallbutfin n \in \omega \[ {b}_{n} \; {\subset}^{\ast} \; {e}^{\psi({\xi}^{\alpha}_{i})}_{\psi \restrict {\xi}^{\alpha}_{i}}\]$, $s$ is a finite set. So we may choose ${k}_{m + 1} \in \omega$ such that ${u}_{\alpha} \cap {X}_{{k}_{m + 1}} \neq 0$ and such that ${k}_{m + 1} > n$ for all $n \in s$. Observe that since ${k}_{m} \in s$, and so ${k}_{m + 1} > {k}_{m}$. Now, ${i}_{m + 1}$ is defined to be the least $i \in \omega$ such that ${\xi}^{\alpha}_{i} \in {X}_{{k}_{m + 1}}$. Since ${k}_{m + 1} \notin s$, ${i}_{m + 1} > {i}_{m}$. Notice that each ${i}_{m}$ is defined so that ${\xi}^{\alpha}_{{i}_{m}} \in {X}_{{k}_{m}}$ and $\forall i < {i}_{m} \[{\xi}^{\alpha}_{i} \notin {X}_{{k}_{m}}\]$. It follows that for each $m \in \omega$
\begin{align*}
	\lc {b}_{{k}_{m}} \cap {e}^{1 - \psi({\xi}^{\alpha}_{{i}_{m}})}_{\psi \restrict {\xi}^{\alpha}_{{i}_{m}}} \cap \left( {\bigcap}_{i < {i}_{m}}{{e}^{\psi({\xi}^{\alpha}_{i})}_{\psi \restrict {\xi}^{\alpha}_{i}}} \right) \rc = \omega. \tag{${\ast}$}
\end{align*}     
Next, choose $\gamma \in {S}_{\alpha}$ such that $\existsinf m \in \omega \[{i}_{m} \in {e}^{0}_{\gamma}\]$ and $\existsinf m \in \omega \[{i}_{m} \in {e}^{1}_{\gamma}\]$. Note that $\gamma \geq \alpha$. Put $\eta = \psi \restrict \gamma$. It follows from $(\ast)$ that for each $m \in \omega$, $\lc {b}_{{k}_{m}} \cap {c}^{\eta}_{{i}_{m}} \rc = \omega$. Therefore, $\existsinf m \in \omega \[\lc {b}_{{k}_{m}} \cap  {e}^{0}_{\eta} \rc = \omega\]$. On the other hand, since ${c}^{\eta}_{i}$ and ${c}^{\eta}_{j}$ are disjoint whenever $i \neq j$, we also get $\existsinf m \in \omega \[\lc {b}_{{k}_{m}} \cap  {e}^{1}_{\eta} \rc = \omega\]$. But this contradicts our initial hypothesis about $\psi$, and we are done.
\end{proof}
Observe that Lemma \ref{lem:nocofinalbranch} is not saying that $\langle {e}_{\psi \restrict \gamma}: \gamma < \kappa \rangle$ is an ${\s}_{\omega, \omega}$ family, for each $\psi \in {2}^{\kappa}$. That would prove $\s = {\s}_{\omega, \omega}$, given $\kappa = \s = \b$ and $P(\kappa)$. For this, we would need $\gamma < \kappa$ so that $\existsinf n \in \omega \[\lc {b}_{n} \cap {e}^{1 - \psi(\gamma)}_{\psi \restrict \gamma}\rc = \omega\]$ and $\existsinf n \in \omega \[\lc {b}_{n} \cap {e}^{\psi(\gamma)}_{\psi \restrict \gamma}\rc = \omega\]$, which is not proved. But Lemma \ref{lem:nocofinalbranch} is still good enough for proving the following analogue of Lemma \ref{lem:positivesplitting}.   
\begin{Lemma} \label{lem:positivesplitting1}
	Let $\A \subset \cube$ be an a.d.\ family. Let $b \in {\I}^{+}(\A)$, and let $\eta \in {2}^{< \kappa}$. Assume that $\forall \beta < \dom{(\eta)} \[b \cap {e}^{1 - \eta(\beta)}_{\eta \restrict \beta} \notin {\I}^{+}(\A) \]$. Then there is a $\tau \in {2}^{< \kappa}$ with $\tau \supset \eta$ such that
\begin{enumerate}
	\item
		$\forall \beta < \dom{(\tau)}\[b \cap {e}^{1 - \tau(\beta)}_{\tau \restrict \beta} \notin {\I}^{+}(\A)\]$.
	\item
		$b \cap {e}^{0}_{\tau} \in {\I}^{+}(\A)$ and $b \cap {e}^{1}_{\tau} \in {\I}^{+}(\A)$.
\end{enumerate} 
\end{Lemma}
\begin{proof}
	Suppose not. In other words, assume that for any $\tau \in {2}^{< \kappa}$, if $\tau \supset \eta$ and if $\forall \beta < \dom{(\tau)}\[b \cap {e}^{1 - \tau(\beta)}_{\tau \restrict \beta} \notin {\I}^{+}(\A)\]$, then there is an $i \in 2$ such that $b \cap {e}^{i}_{\tau} \notin {\I}^{+}(\A)$. This allows us to build a $\psi \in {2}^{\kappa}$ with $\eta \subset \psi$ and with the property that $\forall \beta < \kappa \[b \cap {e}^{1 - \psi(\beta)}_{\psi \restrict \beta} \notin {\I}^{+}(\A)\]$. Now, there exists a collection $\{{b}_{n}: n \in \omega\} \subset {\[b\]}^{\omega}$ with the property that for any $c \in \cube$, if $c$ has infinite intersection with infinitely many ${b}_{n}$, then $c \in {\I}^{+}(\A)$. Applying Lemma \ref{lem:nocofinalbranch} to $\psi$ and $\{{b}_{n}: n \in \omega\}$, we get a $\gamma < \kappa$ such that $\existsinf n \in \omega \[\lc {b}_{n} \cap {e}^{1 - \psi(\gamma)}_{\psi \restrict \gamma}\rc = \omega \]$. But since ${b}_{n} \subset b$, we have that $\existsinf n \in \omega \[\lc {b}_{n} \cap b \cap {e}^{1 - \psi(\gamma)}_{\psi \restrict \gamma}\rc = \omega \]$. It follows that $b \cap {e}^{1 - \psi(\gamma)}_{\psi \restrict \gamma} \in {\I}^{+}(\A)$, contradicting the way we constructed $\psi$.  
\end{proof}
The next definition specifies the analogue of ${\I}_{\eta}$ in the present context. It is simply the obvious modification of ${\I}_{\eta}$.
\begin{Def} \label{def:jeta}
	For any $\eta \in {2}^{< \kappa}$, we define
	\begin{align*}
		{J}_{\eta} = \left\{ \vec{C} \in \C: \forall \gamma < \dom{(\eta)} \forallbutfin n \in \omega \[\vec{C}(n) \subset {e}^{\eta(\gamma)}_{\eta \restrict \gamma}\]\right\}.
	\end{align*}
\end{Def}
The next lemma proves the analogue of Lemma \ref{lem:Ietasplitting}. That $\kappa =\b$ is important here.
\begin{Lemma}\label{lem:columnsplitting1}
	Let $\vec{C}$ be a sequence of columns and let $\eta \in {2}^{< \kappa}$. Assume $\vec{C} \in {J}_{\eta}$. Then there exists $\tau \in {2}^{< \kappa}$ with $\tau \supset \eta$ and $\vec{D} \prec \vec{C}$ such that
	\begin{enumerate}
		\item
			$\vec{D} \in {J}_{\tau}$
		\item
			$\existsinf n \in \omega \[\lc \vec{D}(n) \cap {e}^{0}_{\tau} \rc = \omega\]$ and $\existsinf n \in \omega \[\lc \vec{D}(n) \cap {e}^{1}_{\tau} \rc = \omega\]$ 
	\end{enumerate}
\end{Lemma}
\begin{proof}
	Suppose not. In other words, for any $\tau \in {2}^{< \kappa}$, \emph{if} $\tau \supset \eta$, and \emph{if} there exists a $\vec{D} \prec \vec{C}$ with $\vec{D} \in {J}_{\tau}$, \emph{then} there is an $i \in 2$ such that $\forallbutfin n \in \omega \[\lc \vec{D}(n) \cap {e}^{i}_{\tau} \rc < \omega \]$. Now, construct a $\psi \in {2}^{\kappa}$ with the property that for each $\gamma < \kappa$,
\begin{align*}
	\forallbutfin n \in \omega \[\lc \vec{C}(n) \cap {e}^{1 - \psi(\gamma)}_{\psi \restrict \gamma} \rc < \omega \], \tag{${\ast}_{\gamma}$}
\end{align*}
contradicting Lemma \ref{lem:nocofinalbranch}. To see that this can be done, put $\psi \restrict \dom{(\eta)} = \eta$, and suppose that for some $\dom{(\eta)} \leq \gamma < \kappa$, $\psi\restrict \gamma$ has been defined so that $({\ast}_{\beta})$ holds for each $\beta < \gamma$. Since $\gamma < \kappa = \b$, we can find $\vec{D} \prec \vec{C}$ with $\vec{D} \in {J}_{\psi \restrict \gamma}$ and with the property that $\forall n \in \omega \[\vec{D}(n) {=}^{\ast} \vec{C}(n)\]$. So by the hypothesis there is $i \in 2$ so that $\forallbutfin n \in \omega \[\lc \vec{D}(n) \cap {e}^{i}_{\psi \restrict \gamma} \rc < \omega\]$.  But since $ \vec{D}(n) {=}^{\ast} \vec{C}(n)$ for all $n \in \omega$, if we set $\psi (\gamma) = 1 - i$, then $\psi$ will be as needed.   
\end{proof}
\begin{proof}[Proof of Theorem \ref{thm:main2} (continued)]
Armed with Lemmas \ref{lem:positivesplitting1}, \ref{lem:columnsplitting1}, proceed exactly as in Theorem \ref{thm:main}. At a stage $\alpha < \c$, ${\A}_{\alpha} = \langle {a}_{\beta}: \beta < \alpha \rangle$, $\langle {\T}_{\beta}: \beta < \alpha \rangle$, ${\T}^{\alpha} $, $\langle {\vec{D}}^{\beta}: \beta < \alpha \rangle$ are all exactly as before. Now the nodes $\eta({a}_{\beta})$ and $\eta({\vec{D}}^{\beta}(n))$ satisfy
 \begin{align*}
	& {\vec{D}}^{\beta} \in {J}_{\eta({a}_{\beta})}  \tag{${\dagger\dagger}_{{a}_{\beta}}$}\\ 
   	& \forall \gamma < \dom{({\eta}({\vec{D}}^{\beta}(n)))} \[{\vec{D}}^{\beta}(n) {\subset}^{\ast} {e}^{\eta({\vec{D}}^{\beta}(n))(\gamma)}_{\eta({\vec{D}}^{\beta}(n)) \restrict \gamma}\]. \tag{${\dagger\dagger}_{{\vec{D}}^{\beta}(n)}$}
    \end{align*}  
Given any $b \in {\I}^{+}({\A}_{\alpha})$, apply Lemma \ref{lem:positivesplitting1} to construct $\{{\sigma}_{s}: s \in {2}^{< \omega} \} \subset {2}^{< \kappa}$, $\{{b}_{s}: s \in {2}^{< \omega}\} \subset {\I}^{+}(\A)$, and $\{{\gamma}_{s}: s \in {2}^{< \omega}\} \subset \kappa$ such that
	\begin{enumerate}
		\item
			$\forall s \in {2}^{< \omega} \forall i \in 2 \[\dom{({\sigma}_{s})} = {\gamma}_{s} \wedge {\sigma}_{{s}^{\frown}{\langle i \rangle}} \supset {{\sigma}_{s}}^{\frown}{\langle i \rangle}\]$
		\item
			$\forall s \in {2}^{< \omega} \forall i \in 2 \forall \gamma < \dom{({\sigma}_{s})} \[{b}_{s} \cap {e}^{1 - {\sigma}_{s}(\gamma)}_{{\sigma}_{s} \restrict \gamma} \notin {\I}^{+}({\A}_{\alpha}) \wedge {b}_{{s}^{\frown}{\langle i \rangle}} = {b}_{s} \cap {e}^{i}_{{\sigma}_{s}}\]$
		\item
			${b}_{0} = b$ and $\forall s \in {2}^{< \omega} \[{b}_{s} \cap {e}^{0}_{{\sigma}_{s}} \in {\I}^{+}({\A}_{\alpha}) \wedge {b}_{s} \cap {e}^{1}_{{\sigma}_{s}} \in {\I}^{+}({\A}_{\alpha})\]$.
	\end{enumerate}  
If ${\T}^{\alpha} \subset \T$ is any subtree of ${2}^{< \kappa}$ with $\lc \T \rc < \c$, there is a $f \in {2}^{\omega}$ such that $\tau = {\bigcup}_{n \in \omega}{{\sigma}_{f \restrict n}} \notin \T$. Also, there is ${c}_{0} \in {\[b\]}^{\omega} \cap {\I}^{+}({\A}_{\alpha})$ such that ${c}_{0} \; {\subset}^{\ast} \; {b}_{f \restrict n}$ for all $n \in \omega$. Note that if $\delta < \gamma  = \sup\{{\gamma}_{f \restrict n}: n \in \omega\}$, then $\delta < {\gamma}_{f \restrict n}$ for some $n \in \omega$, and so by (2), ${b}_{f \restrict n} \cap {e}^{1 - \tau(\delta)}_{\tau \restrict \delta} \notin {\I}^{+}({\A}_{\alpha})$. But since ${c}_{0} \; {\subset}^{\ast} \; {b}_{f \restrict n}$, ${c}_{0} \cap {e}^{1 - \tau(\delta)}_{\tau \restrict \delta} \notin {\I}^{+}({\A}_{\alpha})$. Now, proceed exactly as in the proof of Lemma \ref{lem:existsperfecttree} to find $c \in {\[{c}_{0}\]}^{\omega}$ which is a.d.\ from everything in ${\A}_{\alpha}$ and with the property that $\forall \delta < \gamma \[c \; {\subset}^{\ast} \; {e}^{\tau(\delta)}_{\tau \restrict \delta} \]$ (in the present situation $\cf(\kappa) \neq \omega$; so it is obvious that $\gamma < \kappa$).

	Therefore, given $\{{b}_{n}: n \in \omega \} \subset {\I}^{+}({\A}_{\alpha})$, proceed as in the proof of Theorem \ref{thm:main} to find ${c}_{n} \in {\[{b}_{n}\]}^{\omega}$ and ${\tau}_{n} \in {2}^{< \kappa}$ so that each ${c}_{n}$ is a.d.\ from ${\A}_{\alpha}$, ${\tau}_{n} \neq {\tau}_{m}$ and ${c}_{n} \cap {c}_{m} = 0$ whenever $n \neq m$, and $\forall \delta < \dom{({\tau}_{n})}\[{c}_{n} \; {\subset}^{\ast} \; {e}^{\tau(\delta)}_{{\tau}_{n} \restrict \delta} \]$. Put $\vec{{E}_{0}}(n) = {c}_{n}$ and use Lemma \ref{lem:columnsplitting1} to define sequences $\langle {\sigma}_{s}: s \in {2}^{< \omega} \rangle \subset {2}^{< \kappa}$, $\{{\gamma}_{s}: s \in {2}^{< \omega}\} \subset \kappa$, $\langle {\vec{E}}_{s}: s \in {2}^{< \omega}\rangle$, and $\langle {\vec{C}}_{s}: s\in {2}^{< \omega}\rangle$ satisfying
\begin{enumerate}
	\item
		$\forall s \in {2}^{< \omega} \forall i \in 2 \[\dom{({\sigma}_{s})} = {\gamma}_{s} \wedge {\sigma}_{{s}^{\frown}{\langle i \rangle}} \supset {{\sigma}_{s}}^{\frown}{\langle i \rangle}\]$
	\item
		$\forall s \in {2}^{< \omega}\[{\vec{C}}_{s} \in {J}_{{\sigma}_{s}} \wedge {\vec{C}}_{s} \prec {\vec{E}}_{s}\]$
	\item
		$\forall s \in \omega \[\existsinf n \in \omega \[\lc {\vec{C}}_{s}(n) \cap {e}^{0}_{{\sigma}_{s}} \rc = \omega\] \wedge \existsinf n \in \omega \[\lc {\vec{C}}_{s}(n) \cap {e}^{1}_{{\sigma}_{s}} \rc = \omega\]\]$
	\item
		$\forall s \in {2}^{< \omega} \forall i \in 2 \forall n \in \omega \[{\vec{E}}_{{s}^{\frown}{\langle i \rangle}} (n) = {\vec{C}}_{s}({k}_{n}) \cap {e}^{i}_{{\sigma}_{s}} \]$, where $\langle {k}_{n}: n \in \omega \rangle$ is a strictly increasing enumeration of $\left\{n \in \omega: \lc {\vec{C}}_{s}(n) \cap {e}^{i}_{{\sigma}_{s}} \rc = \omega \right\}$.    
\end{enumerate}
There is $f \in {2}^{\omega}$ so that $\tau = {\bigcup}_{n \in \omega}{{\sigma}_{f \restrict n}} \notin \T$, where $\T = {\T}^{\alpha} \cup \{{\tau}_{n} \restrict \delta: n < \omega \wedge \delta \leq \dom{({\tau}_{n})}\}$. Applying Lemma \ref{lem:prec} (which is still true in the present context) to $\langle {\sigma}_{f \restrict n}: n \in \omega \rangle$, $\langle {\gamma}_{f \restrict n}: n \in \omega \rangle$, and $\langle {\vec{C}}_{f \restrict n}: n \in \omega \rangle$, find $\vec{E} \in {J}_{\tau}$ with $\vec{E} \prec {\vec{C}}_{0} \prec {\vec{E}}_{0}$. The rest of the verification is exactly as in the proof of Theorem \ref{thm:main}.
\end{proof}
\section{Some open questions} \label{sec:questions}
At one point in the proof of Lemma \ref{lem:existsperfecttree}, the possibility that $\cf{({\s}_{\omega, \omega})} = \omega$ had to be considered and treated somewhat differently. But we don't know if this case can actually occur. It is a well known open problem whether $\s$ can be singular, but it is easy to see that $\cf{(\s)} \neq \omega$. However, the argument for this doesn't seem to work for ${\s}_{\omega, \omega}$. Brendle~\cite{B} has used a template iteration to produce a model with $\cf{(\a)} = {\aleph}_{\omega}$. We don't know whether this can be modified to work for ${\s}_{\omega, \omega}$.   
\begin{Question} \label{q:cfsww}
Is it consistent that $\cf{({\s}_{\omega, \omega})} = \omega$?
\end{Question}
\begin{Question} \label{q:sww=s}
Does ${\s}_{\omega, \omega} = \s$?
\end{Question}
Of course, if the answer to Question \ref{q:cfsww} is ``yes'', then that would be a dramatic way to show the consistency of ${\s}_{\omega, \omega} \neq \s$. However, we suspect that the answer to Question \ref{q:sww=s} is actually ``yes''. If ${\s}_{\omega, \omega} \neq \s$, then, by Theorem \ref{thm:whenbisbig}, $\b \leq \s$. When $\s = \b$ and $P(\s)$ holds, note that the proof of Theorem \ref{thm:main2} is producing a tree of height $\s$ with the property that the sets along each cofinal branch behave like an ${\s}_{\omega, \omega}$ family, though they may not constitute such a family. At least in the case when $\s \leq \b < {\aleph}_{\omega}$, we are willing to conjecture that $\s = {\s}_{\omega, \omega}$.

	Shelah's construction works by comparing $\s$ with $\a$, while we have compared $\s$ with $\b$. We don't know if $\a$ can replace $\b$ in our construction, but we suspect not.
\begin{Question} \label{q:avsb}
	Suppose $\s \leq \a < {\aleph}_{\omega}$, then is there a weakly tight family?
\end{Question}
Though we have established the analogues of Shelah's cases 1 and 2 for weakly tight families, we have not been able to do this for his case 3. This would require showing that weakly tight families exist when $\b < \s$ provided that some suitable PCF type hypothesis holds, and would imply the existence of such families under $\c < {\aleph}_{\omega}$. But we doubt whether this can be done even when $\c = {\aleph}_{2}$.        
\begin{conj} \label{c:noweaklytight}
	There is a model of $ {\aleph}_{1} = \b < \s = {\aleph}_{2} = \c$ in which there are no weakly tight families.
\end{conj}
Shelah~\cite{b<s} first established the consistency of $\b < \s$. The method is flexible enough to prove the consistency of both $\a = \b < \s$ and $\b < \a = \s$. The method for proving the consistency of $\a = \b < \s$ can be modified to produce a model of $\b < \s$ where a weakly tight family exists. Assuming $\CH$ in the ground model, it is possible to construct a weakly tight family whose weak tightness is not destroyed by the relevant iteration. However, this weakly tight family will not have size $\c$, and we don't know if there are any of size $\s$ in this model. Later, Brendle~\cite{mob2} found a way to prove the consistency of $\b < \a = \s$ via a c.c.c.\ iteration. We do not know whether weakly tight families exist in either Shelah's or Brendle's model for $\b < \a = \s$.   
\begin{conj} \label{c:sacks}
If $\s \leq \b < {\aleph}_{\omega}$, then there is a Sacks indestructible MAD family.
\end{conj}
As mentioned in Section \ref{sec:intro}, we may assume that $\a = \c$ for proving Conjecture \ref{c:sacks}. The difficulty seems to be in finding the right definition of ${\I}_{\eta}$. We need a definition of ${\I}_{\eta}$ which will allow us to do a fusion argument along a branch of cofinality $\omega$, and hence get the analogue of Lemma \ref{lem:prec}.
\bibliographystyle{amsplain}
\bibliography{Bibliography}

\def\polhk#1{\setbox0=\hbox{#1}{\ooalign{\hidewidth
  \lower1.5ex\hbox{`}\hidewidth\crcr\unhbox0}}}
\providecommand{\bysame}{\leavevmode\hbox to3em{\hrulefill}\thinspace}
\providecommand{\MR}{\relax\ifhmode\unskip\space\fi MR }
\providecommand{\MRhref}[2]{%
  \href{http://www.ams.org/mathscinet-getitem?mr=#1}{#2}
}
\providecommand{\href}[2]{#2}
\begin{thebibliography}{10}

\bibitem{BS}
B.~Balcar, J.~Do{\v{c}}k{\'a}lkov{\'a}, and P.~Simon, \emph{Almost disjoint
  families of countable sets}, Finite and infinite sets, Vol.\ I, II (Eger,
  1981), Colloq. Math. Soc. J\'anos Bolyai, vol.~37, North-Holland, Amsterdam,
  1984, pp.~59--88.

\bibitem{mob2}
J.~Brendle, \emph{Mob families and mad families}, Arch. Math. Logic \textbf{37}
  (1997), no.~3, 183--197.

\bibitem{B}
\bysame, \emph{The almost-disjointness number may have countable cofinality},
  Trans. Amer. Math. Soc. \textbf{355} (2003), no.~7, 2633--2649.

\bibitem{BandY}
J.~Brendle and S.~Yatabe, \emph{Forcing indestructibility of {MAD} families},
  Ann. Pure Appl. Logic \textbf{132} (2005), no.~2-3, 271--312.

\bibitem{ersh}
P.~Erd{\H{o}}s and S.~Shelah, \emph{Separability properties of almost-disjoint
  families of sets}, Israel J. Math. \textbf{12} (1972), 207--214.

\bibitem{FKS}
S.~Fuchino, S.~Koppelberg, and S.~Shelah, \emph{Partial orderings with the weak
  {F}reese-{N}ation property}, Ann. Pure Appl. Logic \textbf{80} (1996), no.~1,
  35--54.

\bibitem{G}
S.~Garc{\'{\i}}a-Ferreira, \emph{Continuous functions between
  {I}sbell-{M}r\'owka spaces}, Comment. Math. Univ. Carolin. \textbf{39}
  (1998), no.~1, 185--195.

\bibitem{Hr2}
M.~Hru{\v{s}}{\'a}k, \emph{M{AD} families and the rationals}, Comment. Math.
  Univ. Carolin. \textbf{42} (2001), no.~2, 345--352.

\bibitem{Hr1}
M.~Hru{\v{s}}{\'a}k and S.~Garc{\'{\i}}a~Ferreira, \emph{Ordering {MAD}
  families a la {K}at\'etov}, J. Symbolic Logic \textbf{68} (2003), no.~4,
  1337--1353.

\bibitem{Ku}
M.~S. Kurili{\'c}, \emph{Cohen-stable families of subsets of integers}, J.
  Symbolic Logic \textbf{66} (2001), no.~1, 257--270.

\bibitem{L}
P.~B. Larson, \emph{Almost-disjoint coding and strongly saturated ideals},
  Proc. Amer. Math. Soc. \textbf{133} (2005), no.~9, 2737--2739.

\bibitem{Mal}
V.~I. Malykhin, \emph{Topological properties of {C}ohen generic extensions},
  Trudy Moskov. Mat. Obshch. \textbf{52} (1989), 3--33, 247.

\bibitem{Mi2}
A.~W. Miller, \emph{Arnie {M}iller's problem list}, Set theory of the reals
  (Ramat Gan, 1991), Israel Math. Conf. Proc., vol.~6, Bar-Ilan Univ., Ramat
  Gan, 1993, pp.~645--654.

\bibitem{svmad}
D.~Raghavan, \emph{Maximal almost disjoint families of functions}, Fund. Math.
  \textbf{204} (2009), no.~3, 241--282.

\bibitem{bulletin}
\bysame, \emph{Almost disjoint families and diagonalizations of length
  continuum}, Bull. Symbolic Logic \textbf{16} (2010), no.~2, 240--260.

\bibitem{vmad}
\bysame, \emph{There is a {V}an {D}ouwen {MAD} family}, Trans. Amer. Math. Soc.
  \textbf{362} (2010), no.~11, 5879--5891.

\bibitem{Sh:935}
S.~Shelah, \emph{Mad families and sane player}, preprint, 0904.0816.

\bibitem{b<s}
\bysame, \emph{On cardinal invariants of the continuum}, Axiomatic set theory
  ({B}oulder, {C}olo., 1983), Contemp. Math., vol.~31, Amer. Math. Soc.,
  Providence, RI, 1984, pp.~183--207.

\end{thebibliography}
\end{document}